\documentclass[12pt]{article}%
\usepackage{amsmath, amsfonts, amsthm, color,latexsym}
\usepackage{amsmath}
\usepackage{amsfonts}
\usepackage{amssymb}
\usepackage{color} 
\usepackage{graphicx}%
\providecommand{\U}[1]{\protect\rule{.1in}{.1in}}
\newtheorem{theorem}{Theorem}[section]
\newtheorem{proposition}[theorem]{Proposition}
\newtheorem{corollary}[theorem]{Corollary}
\newtheorem{example}[theorem]{Example}

\newtheorem{remark}[theorem]{Remark}

\newtheorem{lemma}[theorem]{Lemma}
\newtheorem{final remark}[theorem]{Final Remark}
\newtheorem{definition}[theorem]{Definition}
\begin{document}

\title{\sc A general abstract approach to approximation properties in Banach spaces}
\date{}
\author{ Sonia Berrios\thanks{Supported by Fapemig APQ-04687-10.} ~and Geraldo Botelho\thanks{Supported by CNPq Grant
302177/2011-6 and Fapemig Grant PPM-00326-13.\hfill\newline2010 Mathematics Subject
Classification: 46B28, 46A32, 47L20, 47B10, 46G25, 47L22. \newline Keywords: approximation property, operator ideal, Banach space, projective tensor product.} }
\maketitle

\begin{abstract} We propose a unifying approach to many approximation properties studied in the literature from the 1930s up to our days. To do so, we say that a Banach space $E$ has the $({\cal I}, {\cal J}, \tau)$-approximation property if $E$-valued operators belonging to the operator ideal $\cal I$ can be approximated, with respect to the topology $\tau$, by operators belonging to the operator ideal $\cal J$. Restricting $\tau$ to a class of linear topologies, which we call {\it ideal topologies}, this concept recovers many classical/recent approximation properties as particular instances and several important known results are particular cases of more general results that are valid in this abstract framework.
\end{abstract}

\section{Introduction and background}
Aware of the fact that norm limits of finite rank bounded operators in Banach spaces are compact, T. W. Hildebrandt \cite{hild} in 1931 asked if the converse is true. According to A. Pietsch \cite[p.\,54]{pietschhistoria}, {\it this was the most important question ever asked in Banach space theory}. Hildebrandt's question and the mention Banach himself made to the approximation property in his book \cite{livrobanach} mark the starting point of one of the most long standing and productive lines of research in Functional Analysis, especially in Banach space theory, namely, the study of the approximation property and its variants. From Mazur's problem in the Scottish Book in 1936, passing through Grothendieck's memory \cite{grothendieck} in 1953, the counterexamples due to Enflo \cite{enflo} in 1973, Szankowski  \cite{szankowski} in 1981 and Willis \cite{willis} in 1992 and Casazza's survey \cite{Casazza} in 2001, up to recent striking developments, e.g. Figiel, Johnson and Pe{\l}czy\'nski \cite{fjp} in 2011 (and even very recent ones, e.g., Dineen and Mujica \cite{novo}), the approximation property and its variants have been a permanent source of challenging problems and of inspiration to generations of functional analysts. Quoting A. Pietsch once again, {\it life in Banach spaces with certain approximation properties is much easier} \cite[p.\,287]{pietschhistoria}.

The original problem, that concerns the approximation of compact operators by finite rank operators, led to many further questions that can be divided into two great groups: (i) quantitative refinements of the original problem that led, e.g., to the bounded, metric, uniform, bounded projection, commuting bounded approximation properties (see \cite{Casazza}); (ii) variations of the original problem concerning approximation, in different topologies, of bounded operators by operators belonging to different special classes (not only finite rank operators). In this paper we are concerned with the developments arising from the second trend.

As to the topologies involved in the problem, remember that a Banach space $E$ has the (classical, original) approximation property if (and only if) $E$-valued operators can be approximated, with respect to the compact-open topology, by finite rank operators. Once locally convex non-normed topologies are part of the game, the two following consequences were inevitable: (i) The investigation of the approximation property in locally convex spaces, in particular in spaces of holomorphic functions, a trend that was initiated by Aron and Schottenloher \cite{as} in 1976 and reaches our days with the Dineen and Mujica trilogy \cite{dm1, dm2, dm3} (see also \cite{BB}). (ii) The consideration of the approximation of operators by simpler ones with respect to different (locally convex, or at least linear) topologies in the spaces of linear operators.

The first variant of the classical approximation property (AP) in the line we are interested here is the compact approximation property (CAP), which goes back to Banach's book \cite[p.\,237]{livrobanach}, that regards the approximation by compact operators with respect to the compact-open topology. It was only in 1992 that Willis \cite{willis} proved that AP $\neq$ CAP, and it was a strong motivation for mathematicians to consider the problem of approximation by operators belonging to different classes. By the time of Willis' counterexample, the study of special classes of linear operators had been successfully systematized by A. Pietsch with his theory of Operator Ideals \cite{pietsch} (the first edition of Pietsch's book appeared in 1978). The consideration of problems on the approximation by operators belonging to a given operator ideal was a question of time, and indeed a number of approximation properties (APs) with respect to operator ideals have been studied in the last three decades, see, for example, \cite{BB, delgadooja, Delgado-Pineiro, gw, karnsinha, LassalleTurco, OjaIsrael, reinov1, Reinov, reinov}. Furthermore, many other well explored variants of the AP are somehow related to operator ideals, see, for example, \cite{Bourgain-Reinov, Caliskan1, Caliskan4, delgadoJMAA, galicerlassalleturco, Kim2012, llo, lo,  ojanovo, oja, oja2008, oja2012, Reinov 82, sinhakarn}.

In this paper we propose a unification of the approximation properties determined by operator ideals. Our idea is based on the observation that the APs determined by operator ideals already studied in the literature are usually defined (or characterized) by the possibility of approximating operators belonging to a certain class by operators belonging to a smaller class with respect to a certain prescribed topology (many times the compact-open topology). In our approach operator ideals play the role of the classes of operators and we tried to figure out the conditions a topology should satisfy to be suitable for the study of approximation properties. Suitable in the sense that: (i) it should give rise to APs enjoying the usual properties an AP is expected to enjoy; (ii) the resulting APs should recover the already studied APs (at least many important ones) as particular instances; (iii) results about the already studied APs should be particular cases of more general results in this new environment. Our proposal is the concept of {\it ideal topology} detailed in Definition \ref{def1}. The examples of ideal topologies we provide encompass most of the topologies usually used in the study of the APs (cf. Examples \ref{ex1}, \ref{corcot}, \ref{exideal}). The notion of $({\cal I}, {\cal J}, \tau)$-approximation property, as defined in the abstract, recovers a number of APs studied before. The results we prove in Sections \ref{action} and \ref{ptp} make clear that known results about already studied APs can be extended/generalized to our more general setting. Assembling all this information we believe that ideal topologies furnish a suitable framework to study approximation properties in Banach spaces in a rather unified and general way.

The paper is organized as follows: In Section \ref{idtop} we define ideal topologies, give plenty of examples, and we introduce the notion of $({\cal I}, {\cal J}, \tau)$-approximation property, where ${\cal I},{\cal J}$ are operator ideals and $\tau$ is an ideal topology. Several well studied approximation properties in Banach spaces are shown to be particular instances of this just defined abstract concept. In Section \ref{action} we extend/generalize results from \cite{delgadooja, choikimlee} on APs to the language of $({\cal I}, {\cal J}, \tau)$-APs. To reinforce this unifying feature of our new concept, in Section \ref{ptp} we introduce the notion of projective ideal topology in order to prove that recent results from \cite{erhanpilar, BB} on APs in (symmetric) projective tensor products of Banach spaces are particular instances of much more general results in the context of $({\cal I}, {\cal J}, \tau)$-APs. Of course other APs can be found and many other results can be extended/generalized/rephrased within the realm of $({\cal I}, {\cal J}, \tau)$-APs, but we think the examples/results we provide are enough for ideal topologies and $({\cal I}, {\cal J}, \tau)$-APs to prove their worth.

Throughout the paper $E,E_1, \ldots, E_n, F, G, G_1, \ldots, G_n$ are Banach spaces over $\mathbb{K} = \mathbb{R}$ or $\mathbb{C}$. The closed convex hull of a subset $A$ of a Banach space is denoted by $\overline{\rm co}(A)$. By ${\cal L}(E;F)$ we denote the Banach space of bounded linear operators from $E$ to $F$ endowed with the usual operator norm. Given $u \in {\cal L}(E;F)$ and a bounded subset $A \subseteq E$, we use the standard notation
$$\|u\|_A := \sup_{x \in A}\|u(x)\|. $$
The identity operator on a Banach space $E$ is denoted by ${\rm id}_E$ and the symbol $B_E$ stands for the closed unit ball of $E$. Operator ideals are always considered in the sense of Pietsch \cite{defantfloret, pietsch}. By ${\cal L}$ we denote the ideal of all bounded operators between Banach spaces and by $\cal F$  and $\cal K$ the ideals of finite rank and compact operators, respectively. Given a subset $A$ of a topological space $(X, \tau)$, by $\overline{A}^{\,\tau}$ we mean the closure of $A$ in $X$ with respect to the topology $\tau$.

The space of continuous $n$-linear mappings from $E_1 \times \cdots \times E_n$ to $F$ is denoted by ${\cal L}(E_1, \ldots, E_n;F)$ (${\cal L}(^nE;F)$ if $E_1 = \cdots = E_n = E$), and the space of continuous $n$-homogeneous polynomials from $E$ to $F$ by ${\cal P}(^nE;F)$. Both ${\cal L}(E_1, \ldots, E_n;F)$ and ${\cal P}(^nE;F)$ are Banach spaces with their usual sup norms. The completed $n$-fold projective tensor product of $E_1, \ldots, E_n$ is denoted by $E_1\widehat{\otimes}_{\pi}\cdots  \widehat{\otimes}_{\pi}E_n$, and the completed $n$-fold symmetric projective tensor product of $E$ by $\widehat{\otimes}_{s,\pi}^n E$. An elementary symmetric tensor $x \otimes \stackrel{(n)}{\cdots} \otimes x$ shall be simply denoted by $\otimes^n x$. Given an $n$-linear mapping $A \in {\cal L}(E_1, \ldots, E_n;F)$ and a polynomial $P \in {\cal P}(^nE;F)$, by $A_L$ and $P_L$ we denote their linearizations, that is,
$$A_L \in {\cal L}\left( E_1\widehat{\otimes}_{\pi}\cdots  \widehat{\otimes}_{\pi}E_n;F\right)~,~A_L(x_1 \otimes \cdots \otimes x_n) = A(x_1, \ldots, x_n)~~{\rm and} $$
$$P_L \in {\cal L}\left(\widehat{\otimes}_{s,\pi}^n E ; F\right)~,~P_L(\otimes^n x) = P(x).$$
For background on spaces of multilinear mappings and homogeneous polynomials we refer to \cite{livrodineen, livromujica}, and for (symmetric) projective tensor products of Banach spaces we refer to \cite{defantfloret, klaus, ryan}.

\section{Ideal topologies}\label{idtop}
In this section we define the notion of ideal topology and provide a method to generate many useful examples. The approximation property with respect to a pair of operator ideals and a given ideal topology is defined. We show that many approximation properties studied in the literature arise as particular instances of this general concept. A few basic properties are proved.

\begin{definition}\label{def1}\rm An {\it ideal topology} $\tau$ is a correspondence that, for all Banach spaces $E$ and $F$, assigns a linear topology, still denoted by $\tau$, on the space ${\cal L}(E;F)$ such that: for every operator ideal $\cal I$, if
$$\overline{\cal I}^{\,\tau}(E;F) := \overline{{\cal I}(E;F)}^{\,\tau}$$
for all Banach spaces $E$ and $F$, then $\overline{\cal I}^{\,\tau}$ is an operator ideal.
\end{definition}

\begin{remark}\label{remark1}\rm Let $\cal I$ be an arbitrary operator ideal. Since ${\cal I}(E;F)$ is a linear subspace of ${\cal L}(E;F)$ and $\left({\cal L}(E;F), \tau\right)$ is a topological vector space, it is always true that $\overline{\cal I}^{\,\tau}(E;F)$ is a linear subspace of ${\cal L}(E;F)$. Moreover, it is plain that ${\cal F}(E;F) \subseteq \overline{\cal I}^{\,\tau}(E;F)$. So, once a linear topology is assigned to each of the spaces ${\cal L}(E;F)$, the ideal property of $\overline{\cal I}^{\,\tau}$ is all that has to be checked to show that $\tau$ is an ideal topology.
\end{remark}

\begin{example}\label{ex1}\rm (a) It is folklore that the norm topology, which is the topology of uniform convergence on bounded sets, denoted by $\| \cdot \|$, is an ideal topology.\\
(b) The topology of pointwise convergence $\tau_P$, which is the topology of uniform convergence on finite sets, is an ideal topology. Indeed, the topology $\tau_P$ is linear because it is the locally convex topology generated by the seminorms ported by finite sets (or, equivalently, by singletons). It is straightforward to check that $\overline{\cal I}^{\,{\tau}_p}$ is an operator ideal for every operator ideal $\cal I$. The topology $\tau_P$ is sometimes refereed to as the strong operator topology (SOT).
\end{example}

Now we give a method to generate ideal topologies ranging from $\tau_P$ to $\|\cdot\|$. By BAN we denote the class of all Banach spaces over $\mathbb{K}$.

\begin{proposition} \label{propger} Suppose that for every Banach space $E$ it has been assigned a collection ${\cal A}(E)$ of bounded subsets of $E$ such that $\{x\} \in {\cal A}(E)$ for every $x \in E$ and
\begin{equation}\label{cond}u(A) \in {\cal A}(F)~{\rm for ~all~}E,F \in {\rm BAN}, A \in {\cal A}(E) {\rm ~and~} u \in {\cal L}(E;F).
\end{equation}
Then the topology $\tau_{\cal A}$ of uniform convergence on sets belonging to ${\cal A}(E)$, $E \in {\rm BAN}$, is an ideal topology. Moreover, $\tau_P \subseteq \tau_{\cal A} \subseteq \|\cdot\|$.
\end{proposition}

\begin{proof} First note that $\tau_{\cal A}$ is not the discrete topology on ${\cal L}(E;F)$ as ${\cal A}(E) \neq \emptyset$. So $\tau_{\cal A}$ is a linear topology
because, for all Banach spaces $E$ and $F$, it is the locally convex topology on ${\cal L}(E;F)$ generated by the seminorms
ported by the sets belonging to ${\cal A}(E)$, that is, by the seminorms
$$u \in {\cal L}(E;F) \mapsto \|u\|_A := \sup_{x \in A}\|u(x)\|, $$
where $A \in {\cal A}(E)$. Let $\cal I$ be an operator ideal. By Remark \ref{remark1} we just have to check that $ \overline{\cal
I}^{\,\tau_{\cal A }}$ enjoys the ideal property. Given operators $u \in {\cal
L}(E;F)$, $v \in \overline{\cal I}^{\,\tau_{A}}(F;G)$,
$0 \neq w \in {\cal L}(G;H)$, a subset $A$ of $E$ belonging to ${\cal A}(E)$
and $\varepsilon
> 0$, by (\ref{cond}) we know that $u(A) \in {\cal A}(F)$, so we can take an operator $T \in {\cal I}(F;G)$ such that $\|v - T\|_{u(A)} <
\frac{\varepsilon}{\|w\|}$. Then $w \circ T \circ u \in {\cal
I}(F;G)$ by the ideal property of $\cal I$ and
$$\|w \circ v \circ u - w \circ T \circ u\|_A \leq \|w\| \cdot \|v - T\|_{u(A)} < \varepsilon,$$
proving that $w \circ v \circ u \in  \overline{\cal
I}^{\,\tau_{\cal A}}(E;H)$. The second assertion is obvious because ${\cal A}(E)$ contains the singletons and is contained in the set of all bounded subsets of $E$.
\end{proof}

\begin{remark}\rm In order to have a linear topology, we have to avoid the discrete topology on ${\cal L}(E;F)$. This was done with the condition that ${\cal A}(E)$ contains the singletons. Of course this could have been done in many different ways, but the containment of the singletons also implies that $\tau_P \subseteq \tau_{\cal A}$. In Proposition \ref{tcpontual} the reader will understand why we are restricting ourselves to ideal topologies containing $\tau_P$.
\end{remark}

Proposition \ref{propger} allows us to show that several well known and useful topologies can be found in our way from $\tau_P$ to $\|\cdot\|$:

\begin{example}\label{corcot}\rm It is plain that bounded linear operators send compact sets to compact sets, so the compact-open topology $\tau_c$, which is the topology of uniform convergence on compact sets, is an ideal topology. The same happens for the following classes of subsets of Banach spaces: compact and convex sets, weakly compact sets, weakly compact and convex sets (remember that bounded linear operators are weak-weak continuous). So the topologies of uniform convergence on sets belonging to each of these classes are ideal topologies.
\end{example}

We need the following terminology to use Proposition \ref{propger} to give more useful examples of ideal topologies. Given an operator ideal $\cal I$ and a Banach space $E$, according to \cite{Stephani, GG, LassalleTurco} we define
$$C_{\cal I}(E) = \{ A \subseteq E : \exists F, \exists u \in {\cal I}(F;E) {\rm ~such~that~} A \subseteq u(B_F)\}, $$
$$K_{\cal I}(E) = \{\overline{A} :  A \subseteq E, \exists F, \exists K \subseteq F {\rm~ compact},\exists u \in {\cal I}(F;E) {\rm ~such~that~} A \subseteq u(K)\}.$$
The sets belonging to $C_{\cal I}(E)$ are called {\it $\cal I$-bounded sets} and the sets belonging to $K_{\cal I}(E)$ are called {\it $\cal I$-compact sets}.

\begin{example}\label{exideal}\rm Let $\cal I$ be an operator ideal.\\
(a) It is clear that $\cal I$-bounded sets are norm bounded and that singletons are $\cal I$-bounded (indeed, this is obvious for $x = 0$, and for $x \neq 0$ just pick a funcional $\varphi \in E'$ such that $\varphi(x) = \|x\|$ and note that $\varphi \otimes x \in {\cal I}(E;E)$ and $\varphi \otimes x\left(x/\|x\| \right) = x$). By the ideal property of $\cal I$ it follows that bounded linear operators send $\cal I$-bounded sets to $\cal I$-bounded sets, so the  topology $\tau_{C_{\cal I }}$ of uniform convergence on $\cal I$-bounded sets (cf. e.g., \cite{aronruedaprims}) is an ideal topology by Proposition \ref{propger}.

\medskip

\noindent(b) It is clear that $\cal I$-compact sets are norm bounded (actually they are norm compact) and, like before, that singletons are $\cal I$-compact. By the ideal property of $\cal I$ it follows that bounded linear operators send $\cal I$-compact sets to $\cal I$-compact sets, so the  topology $\tau_{K_{\cal I }}$ of uniform convergence on $\cal I$-compact sets (cf. e.g., \cite{LassalleTurco, Delgado-Pineiro}) is an ideal topology by Proposition \ref{propger}. In particular, the topology $\tau_{K_p}$ of uniform convergence on $p$-compact sets (cf. e.g., \cite{sinhakarn}) is an ideal topology. Indeed, if ${\cal K}_p$ denotes the ideal of $p$-compact operators, then $\tau_{K_p} = \tau_{K_{{\cal K}_p}}.$

\medskip

\noindent(c) For $q > 0$, a subset $A$ of a Banach space $E$ is a {\it Bourgain-Reinov $q$-compact set} (see \cite{Bourgain-Reinov, Reinov, Ain-Lillimets-Oja}), in symbols $A \in BR_q(E)$, if there is a $E$-valued absolutely $q$-summable sequence $(x_n)_n$ such that $A$ is contained in the closure of the absolutely convex hull of $\{x_1, x_2, \ldots,\}$. It is clear that Bourgain-Reinov $q$-compact sets are bounded and that singletons are Bourgain-Reinov $q$-compact. Since continuous linear operators send $q$-summable sequences to $q$-summable sequences, it follows that the class $BR_q$ of Bourgain-Reinov $q$-compact subsets of Banach spaces fulfills condition (\ref{cond}). By Proposition \ref{propger} results that the topology $\tau_{{BR}_q}$ of uniform convergence on Bourgain-Reinov $q$-compact sets is an ideal topology. For $q \geq 1$, the sets in $BR_q$ are also called relatively Grothendieck $q$-compact \cite{Kim2012}.
\end{example}

With plenty of useful linear topologies in hands we can define the approximation properties determined by a pair of operators ideals and a given linear topology:

\begin{definition} \rm Let ${\cal I}, {\cal J}$ be operator ideals and $\tau$ be an ideal topology. We say that a Banach space $E$ has the:\\
(a) $({\cal I}, {\cal J}, \tau)$-approximation property, $({\cal I}, {\cal J}, \tau)$-AP for short, if
$${\cal I}(F;E) \subseteq \overline{{\cal J}(F;E)}^{\, \tau} {\rm ~for~every~Banach~space~} F; $$
(b) $({\cal I}, {\cal J}, \tau)$-weak approximation property, $({\cal I}, {\cal J}, \tau)$-WAP for short, if
$${\cal I}(E;E) \subseteq \overline{{\cal J}(E;E)}^{\, \tau}. $$
\end{definition}

The examples below unfold that many well studied approximation properties are particular cases of our general concept. It is good to have in mind the following characterizations, which are immediate consequences of the ideal property of $\overline{{\cal I}}^{\,\tau}$:
\begin{align*} E {\rm ~has~ the~} ({\cal L}, {\cal I}, \tau){\rm-AP} & \Longleftrightarrow {\cal L}(E;E) \subseteq \overline{{\cal I}(E;E)}^{\, \tau}\\ &\Longleftrightarrow {\rm id}_E \in  \overline{{\cal I}(E;E)}^{\, \tau}\\& \Longleftrightarrow E {\rm ~has~ the~} ({\cal L}, {\cal I}, \tau){\rm-WAP}.
\end{align*}
By ${\cal I}^{\rm \,sur}$ we mean the surjective hull of the operator ideal $\cal I$.
\begin{example}\label{primex}\rm
(a) The classical approximation property coincides with the $({\cal K}, {\cal F}, \|\cdot\|)$-AP, with the $({\cal L}, {\cal F}, \tau_c)$-AP (hence with the $({\cal L}, {\cal F}, \tau_c)$-WAP).

\medskip

\noindent(b) The compact approximation property coincides with the $({\cal L}, {\cal K}, \tau_c)$-AP (hence with the $({\cal L}, {\cal K}, \tau_c)$-AP).

\medskip

\noindent(c) Let $\cal I$ be an operator ideal. The $\cal I$-approximation property of \cite{BB} coincides with the $({\cal L}, {\cal I}, \tau_c)$-AP (hence with the $({\cal L}, {\cal I}, \tau_c)$-AP).

\medskip

\noindent(d) The weak approximation property of Choi and Kim \cite{Choi-Kim 2006} coincides with the $({\cal K}, {\cal F},\tau_c)$-WAP.

\medskip

\noindent(e) The quasi approximation property of Choi and Kim \cite{Choi-Kim 2006} coincides with the $({\cal K}, {\cal F},\|\cdot\|)$-WAP.

\medskip

\noindent(f) Let $\cal I$ be an operator ideal. The $\cal I$-approximation property of Lassalle and Turco \cite{LassalleTurco} and the approximation property with respect to the operator ideal $\cal I$ of Delgado and Pi\~neiro \cite{Delgado-Pineiro} both coincide with the $({\cal L}, {\cal F}, \tau_{K_{\cal I}})$-AP (hence with the $({\cal L}, {\cal F}, \tau_{K_{\cal I}})$-WAP) and with the $({\cal I}^{\rm \,sur}, {\cal F}, \tau_c)$-AP (see \cite[Theorem 2.3]{Delgado-Pineiro}) .

\medskip

\noindent(g) The $p$-approximation property of Sinha and Karn \cite{sinhakarn} (see also \cite{delgadooja}), $1 \leq p < \infty$, coincides with the $({\cal L}, {\cal F}, \tau_{K_{{\cal N}_p}})$-AP, where ${\cal N}_p$ is the ideal of $p$-nuclear operators \cite{LassalleTurco} (hence with the $({\cal L}, {\cal F}, \tau_{K_{{\cal N}_p}})$-WAP), with the $({\cal L}, {\cal F}, \tau_{K_{{\cal K}_p}})$-AP), where ${\cal K}_p$ is the ideal of $p$-compact operators (hence with the $({\cal L}, {\cal F}, \tau_{K_{{\cal K}_p}})$-WAP) and with the $({\cal K}_p^{\rm \,sur}, {\cal F}, \tau_c)$-AP (see \cite[Corollary 2.6]{Delgado-Pineiro}).

\medskip

\noindent(h) Let $0 < p \leq 1$, $q = p/(1-p)$ and $BR_q$ be the class of Bourgain-Reinov $q$-compact subsets of Banach spaces (cf. Example \ref{exideal}(c)). The aproximation property of order $p$ of Reinov \cite{Reinov 82} coincides with the $({\cal L}, {\cal F}, \tau_{BR_q})$-AP (hence with the $({\cal L}, {\cal F}, \tau_{BR_q})$-WAP) (see \cite{Bourgain-Reinov, Reinov} and \cite[p.\,70]{Delgado-Pineiro}). For $q \geq1$, the approximation property of order $p$ is also called the Grothendieck $q$-approximation property (see \cite{Kim2012}).

\medskip

\noindent(i) A long standing problem (see \cite[Problem 1.e.9]{LT}) asks whether the classical approximation property coincides with the $({\cal K}, {\cal F}, \|\cdot\|)$-WAP.
\end{example}

The topology $\tau_P$ of pointwise convergence is an extreme case (and this is the reason why we are interested in ideal topologies containing $\tau_P$ -- cf. Proposition \ref{propger}):

\begin{proposition}\label{tcpontual} Regardless of the Banach space $E$ and the operator ideals ${\cal I}$ and $ {\cal J}$, $E$ has the $({\cal I}, {\cal J}, \tau_P)$-AP and the $({\cal I}, {\cal J}, \tau_P)$-WAP.
\end{proposition}

\begin{proof} It is easy to see that, for every Banach space $E$, ${\rm id}_E \in \overline{{\cal F}(E;E)}^{\,\tau_P}$ (see \cite[Proposition 3.14]{teselissitsin}). Since $\overline{\cal F}^{\,\tau_P}$ is an operator ideal, we have $\overline{{\cal F}(F;E)}^{\,\tau_P} = {\cal L}(F;E)$ regardless of the Banach spaces $E$ and $F$. Now the result is immediate.
\end{proof}

Several usual properties of the known approximation properties extend to this more general context. We finish this section showing three examples that illustrate the situation and will be useful later:

\begin{proposition}\label{complem} Let ${\cal I}, {\cal J}$ be operator ideals and $\tau$ be an ideal topology. If the Banach space $E$ has the $({\cal I}, {\cal J}, \tau)$-AP ($({\cal I}, {\cal J}, \tau)$-WAP, respectively) and the Banach space $F$ is isomorphic to a complemented subspace of $E$, then $F$ has the $({\cal I}, {\cal J}, \tau)$-AP ($({\cal I}, {\cal J}, \tau)$-WAP, respectively) as well.
\end{proposition}

\begin{proof} There are continuous linear operators $u \colon F \longrightarrow E$ and $v \colon E \longrightarrow F$ such that $v \circ u = {\rm id}_F$. Let $G$ be a Banach space and $T \in {\cal I}(G;F)$. Then $u \circ T \in {\cal I}(G;E)$, and by the $({\cal I}, {\cal J}, \tau)$-AP of $E$ we know that $u \circ T \in \overline{{\cal J}(G;E)}^{\, \tau}$. By the ideal property of $\overline{{\cal J}}^{\, \tau}$ it follows that $T = v \circ u \circ T \in \overline{{\cal J}(G;F)}^{\, \tau}$, proving that $F$ has the $({\cal I}, {\cal J}, \tau)$-AP. The case of the WAP is analogous.
\end{proof}

\begin{proposition} Let $\tau$ be an ideal topology, ${\cal I}, {\cal J}$ be operator ideals and $E_1, \ldots, E_n$ be Banach spaces. Then the finite direct sum $\bigoplus_{j=1}^n E_j$ has the $({\cal I}, {\cal J}, \tau)$-AP (the $({\cal I}, {\cal J}, \tau)$-WAP, respectively) if and only if $E_j$ has the $({\cal I}, {\cal J}, \tau)$-AP (the $({\cal I}_j, {\cal J}_j, \tau)$-WAP, respectively) for $j = 1, \ldots, n$.
\end{proposition}

\begin{proof} Assume that $E_j$ has the $({\cal I}, {\cal J}, \tau)$-AP for $j = 1, \ldots, n$. For each $j$ let $i_j \colon E_j \longrightarrow \bigoplus_{j=1}^n E_j$ and $q_j \colon \bigoplus_{j=1}^n E_j \longrightarrow E_j$ be the canonical operators. Given a Banach space $F$ and an operator $u \in {\cal I}\left(F; \bigoplus_{j=1}^n E_j \right)$, we have that $q_j \circ u \in {\cal I}\left(F;E_j \right)$, hence $q_j \circ u \in \overline{\cal J}^{\, \tau}\left(F;E_j \right)$. Then each $i_j \circ q_j \circ u \in \overline{\cal J}^{\, \tau}\left(F; \bigoplus_{j=1}^n E_j \right)$, so $u = \sum_{j=1}^n i_j \circ q_j \circ u \in \overline{\cal J}^{\, \tau}\left(F; \bigoplus_{j=1}^n E_j \right)$. The case of the WAP is analogous.
\end{proof}

\begin{proposition} Let ${\cal I}, {\cal J}_1, {\cal J}_2$ be operator ideals and let $\tau_1, \tau_2$ be ideal topologies such that ${\cal J}_1 \subseteq {\cal J}_2 \subseteq \overline{{\cal J}_1}^{\,\tau_1}$ and $\tau_2 \subseteq \tau_1$. Then a Banach space $E$ has the $({\cal I}, {\cal J}_1, \tau_2)$-AP (the $({\cal I}, {\cal J}_1, \tau_2)$-WAP, respectively) if and only if $E$ has the $({\cal I}, {\cal J}_2, \tau_2)$-AP (the $({\cal I}, {\cal J}_2, \tau_2)$-WAP, respectively).
\end{proposition}

\begin{proof} One implication follows immediately from the inclusion ${\cal J}_1 \subseteq {\cal J}_2$ and the reverse implication follows from $$\overline{{\cal J}_2(F;E)}^{\,\tau_2} \subseteq \overline{\left(\overline{{\cal J}_1(F;E)}^{\,\tau_1}\right)}^{\,\tau_2} \subseteq \overline{\left(\overline{{\cal J}_1(F;E)}^{\,\tau_2}\right)}^{\,\tau_2} =\overline{{\cal J}_1(F;E)}^{\,\tau_2}.$$
\end{proof}

\section{Ideal topologies in action}\label{action}
An important aspect of the approximation properties in Banach spaces is the fact that, sometimes, the approximation by two different classes of operators with respect to two different topologies actually coincide. The search for this kind of situation in our case can be rephrased as: When does the equality $({\cal I}_1, {\cal J}_1, \tau_1)$-AP = $({\cal I}_2, {\cal J}_2, \tau_2)$-AP hold? What about the WAP? There are several trivial coincidences, for example:\\
$\bullet$ Let ${\cal I}_1, {\cal I}_2, {\cal J}_1, {\cal J}_2$ be operator ideals and $\tau_1, \tau_2$ be ideal topologies such that ${\cal I}_2 \subseteq {\cal I}_1, {\cal J}_1 \subseteq {\cal J}_2$ and $\tau_2 \subseteq \tau_1$. If a Banach space $E$ has the $({\cal I}_1, {\cal J}_1, \tau_1)$-AP, then $E$ has the $({\cal I}_2, {\cal J}_2, \tau_2)$-AP. The same holds for the corresponding WAP's. \\
$\bullet$ As we have already remarked, for all Banach spaces $E$, operator ideals $\cal J$ and ideal topologies $\tau$, the following are equivalent:\\
(i) ${\rm id}_E \in \overline{{\cal J}(E;E)}^{\, \tau}$,\\
(ii) $E$ has the $({\cal L}, {\cal J}, \tau)$-AP (hence $E$ has the $({\cal I}, {\cal J}, \tau)$-AP for every operator ideal $\cal I$),\\
(iii) $E$ has the $({\cal L}, {\cal J}, \tau)$-WAP (hence $E$ has the $({\cal I}, {\cal J}, \tau)$-WAP for every operator ideal $\cal I$).

The aim of this section is to make clear that the abstract notion of ideal topology is appropriate to detect this kind of coincidence. We prove some  non-trivial coincidences that extend and generalize previous results, mainly from \cite{delgadooja} and \cite{choikimlee}. The argument of the following lemma shall be repeated several times, so we state it separately for further reference.

\begin{lemma} \label{lemmaaderencia}Let ${\cal I}$ be an operator ideal, $E$, $F_1$ and $F_2$ be Banach spaces, ${\cal A}_i$ be a collection of bounded subsets of $F_i$ and $\tau_i$ be the locally convex topology on ${\cal L}(F_i;E)$ generated by the seminorms ported by the sets belonging to ${\cal A}_i$, $i = 1,2$. If $R\in
\overline{\mathcal{I}(F_1;E)}^{\tau_1}$ and $S\in \mathcal{L}(F_2;F_1)$ is such that $S(A) \in {\cal A}_1$ for every $A \in  {\cal A}_2$, then $R \circ S\in
\overline{\mathcal{I}(F_2;E)}^{\tau_2}$.
\end{lemma}

\begin{proof} It is clear that $\tau_i$ is the topology of uniform convergence on the sets belonging to ${\cal A}_i$. Let $\varepsilon>0$ and $A \in {\cal A}_2$ be given. By assumption we have $S(A) \in {\cal A}_1$ and $R\in
\overline{\mathcal{I}(F_1;E)}^{\tau_1}$, so there exists an
operator $T\in {\cal I}(F_1;E)$ such that
$$\|T\circ S-R\circ S\|_{A} = \|T-R\|_{S(A)}<\varepsilon.$$
Since $T\circ S\in
{\cal I}(F_2;E)$ it follows that $R\circ S\in \overline{{\cal I}(F_2;E)}^{\tau_2}$.
\end{proof}


\begin{definition}\rm \cite[p.\,962]{aronruedaprims} Let $\cal I$ be an operator ideal. An operator $T \in {\cal L}(E;F)$ is said to be {\it $\cal I$-bounded} if for every $x\in E$
there exists a neighborhood $V_x$ of $x$ such that $T(V_x)\in
C_{\cal I}(F)$. The set of all $\cal I$-bounded linear operators from $E$ to $F$ is denote by $\mathcal{L}_{\cal I}(E;F)$.
\end{definition}

It is well known that (see \cite{aronruedaprims}):\\
$\bullet$ $\mathcal{L}_{\cal I}$ is an operator ideal.\\
$\bullet$ $T \in {\cal L}_{\cal I}(E;F)$ if and only if $T(B_E)\in C_{\cal
I}(F)$.\\
$\bullet$ $\mathcal{I}(E;F)\subseteq \mathcal{L}_{\cal I}(
E;F).$

Some of the implications of Delgado, Oja, Pi\~neiro and Serrano \cite[Theorem 2.1]{delgadooja} hold true in a rather general context:

\begin{proposition}\label{prop00} Let $E$ be a Banach space and let
 ${\cal I}_1,{\cal I}_2,{\cal I}_3, {\cal J}_1 $ and ${\cal J}_2 $ be operator ideals such that  ${\cal J}_2 \subseteq {\cal J}_1 $
  and ${\cal I}_1 \subseteq {\cal I}_3 \cap \left({\cal I}_1 \circ \mathcal{L}_{\mathcal{J}_2}\right)$.
  Consider the following conditions:\\
{\rm (a)} ${\rm id}_E \in \overline{\mathcal{I}_2(E;E)}^{\tau_{C_{\mathcal {J}_1}}}$.\\
{\rm (b)} $E$ has the $({\cal I}_1, {\cal I}_2, \|\cdot\|)$-AP.\\
{\rm (c)} $E$ has the $({\cal
I}_1, {\cal I}_2,\tau_{C_{\mathcal {J}_2}} )$-AP.\\
{\rm (d)} $E$ has the $({\cal
I}_3, {\cal I}_2,\tau_{C_{\mathcal {J}_2}} )$-AP.\\
{\rm (e)} $E$ has the $({\cal
L}, {\cal I}_2,\tau_{C_{\mathcal {J}_2}} )$-AP.\\
Then ${\rm (a)} \Longrightarrow {\rm (e)} \Longrightarrow {\rm (d)} \Longrightarrow {\rm (c)}\Longleftrightarrow {\rm (b)}$.
\end{proposition}

\begin{proof}
(a) $\Longrightarrow$ (e) Let $F$ be a Banach space and $T\in \mathcal
{L}(F;E)$. Since  ${\cal J}_2 \subseteq {\cal J}_1 $, $T$ maps $\mathcal{J}_2$-bounded sets to
 $\mathcal{J}_1$-bounded sets. By (a) and  Lemma
\ref{lemmaaderencia} we have $T = {\rm id}_E \circ T \in
\overline{\mathcal{I}_2(F;E)}^{\tau_{C_{\mathcal {J}_2}}} $. Therefore $E$ has the $({\cal L}, {\cal
I}_2,\tau_{C_{\mathcal {J}_2}} )$-AP. \\
(e) $\Longrightarrow $ (d) $\Longrightarrow $ (c)  are obvious.\\
(c) $\Longrightarrow $(b) Let $F$ be a Banach space and $T\in \mathcal
{I}_1(F;E)$. There are a Banach space $G$ and operators $R\in {\cal I}_1(G;E)$ and $S\in \mathcal{L}_{\mathcal{J}_2}(F;G)$ such that  $ T= R\circ S$. Then $R \in \overline{\mathcal{I}_2(G;E)}^{\tau_{C_{\mathcal {J}_2}}}$ and $S$ maps bounded sets to $\mathcal{J}_2$-bounded sets. By Lemma
\ref{lemmaaderencia} we have $T = R \circ S\in \overline{\mathcal{I}_2(F;E)}^{\|\cdot\|}$, proving that $E$ has the $({\cal I}_1, {\cal
I}_2,\|\cdot \| )$-AP.\\
(b) $\Longrightarrow $(c) From Example \ref{exideal} and Proposition \ref{propger} we have $\tau_{C_{{\cal J}_2}}\subseteq \|\cdot\| $. Thus the $({\cal I}_1, {\cal I}_2,
\|\cdot\|)$-AP implies the $({\cal I}_1, {\cal I}_2,\tau_{C_{{\cal J}_2}}
)$-AP .
\end{proof}

\begin{remark}\rm Note that in Proposition \ref{prop00} no condition has been imposed on the operator ideal ${\cal I}_2$.
\end{remark}

The aim now is to show that, under some additional assumptions, the conditions (a)-(d) above are all equivalent. To accomplish this task we take advantage of the quantitative change Lima, Nygaard and Oja \cite{OjaIsrael} made in the classical Davis, Figiel, Johnson and Pe{\l}czy\'nski classical factorization scheme \cite{DFJP}, which we describe next.

Let $E$ be a Banach space, let $K$ be a closed absolutely convex
subset of its unit ball $B_E$ and let $a>1$. For each $n \in \mathbb{N}$
put $B_n=a^{n/2}K+a^{-n/2}B_E$. As $B_n$ is absolutely convex and absorbent, the gauge (Minkowski functional) $\|\cdot\|_n$ of $B_n$,
$$\|x \|_n=\inf \{\lambda : x\in \lambda B_n\},$$  is a seminorm on
$E$ that is equivalent to the original norm $\|\cdot\|$ on $E$. For $x \in E$ define $\|x\|_K=\left(\sum_{n=1}^{\infty} \|x\|_n^2 \right)^{1/2}$ and let the subspace $E_K=\{x\in E: \|x\|_K<\infty\}$ of $E$ be endowed with the norm $\| \cdot\|_K$. The function $$f\colon (1, \infty)\longrightarrow \mathbb{R}~,~f(a)=\sum_{n=1}^{\infty } \dfrac{a^n}{(a^n+1)^2},$$
is continuous, strictly decreasing, $\displaystyle\lim_{a\rightarrow
1^+}f(a)=\infty $ and $\displaystyle\lim_{a\rightarrow \infty}
f(a)=0 $. Hence, there exists exactly one number $\hat{a}\in (1,
\infty)$ such that $f(\hat{a})=1$. Let $C_K=\{x\in E:
\|x\|_K\leq 1\}$ and let $J_K$ denote the identity embedding from $E_K$
to $E$. Replacing $a$ with $\hat a$ in \cite[Lemma 1.1]{OjaIsrael}, we get

\begin{lemma}\label{lnofl} {\rm \cite[Lemma 1.1]{OjaIsrael}}
Let $E$, $K$, $C_K$, $E_K$ and $J_K$ be as
above.  Then:\\
{\rm (a)} $K\subseteq C_K \subseteq B_E$. \\
{\rm (b)} $E_K$ is a Banach space with closed unit ball $C_K$ and
$J_K\in \mathcal{L}(E_K;E)$ with
$\|J_K\|\leq 1$.\\
{\rm (c)} $J_K''$ is injective. \\
{\rm (d)} $J_K(C_K)=C_K$.
\end{lemma}

%

The key result is the following:

\begin{theorem} {\rm (Lima-Nygaard-Oja Factorization Theorem \cite[Theorem 2.2]{OjaIsrael})} Suppose
$T\in \mathcal{L}(F;E)$. Let $K=\dfrac{1}{\|T\|}\overline{T(B_F)}$
and let $T_K\in\mathcal{L}(F; E_K)$ be defined by $T_K(y)=T(y), y\in
F$. Then $T=J_K\circ T_K$.
\end{theorem}

Henceforth the expression $T=J_K\circ T_K$ above shall be referred to as {\it the LNO factorization of $T$}.

\begin{definition}\rm An operator ideal $\cal I$ has the {\it Grothendieck property} if whenever $A$ is a bounded subset of a Banach space $E$ such that for every
$\varepsilon>0$ there is a set $A_{\varepsilon}\in C_{\cal I}(E)$ with $A\subseteq A_{\varepsilon}+\varepsilon B_E$, it holds that $A\in C_{\cal
I}(E)$.
\end{definition}

\begin{example}\label{exGG}\rm Gonz\'alez and Guti\'errez \cite[Proposition 3(c)]{GG} proved that any
closed surjective operator ideal has the Grothendieck property. Lists of closed surjective operator ideals can be found in \cite{GG} and \cite{DJP}.
\end{example}

\begin{proposition}\label{lemmajk} Let $T=J_K\circ T_K$ be the LNO factorization of the operator $T\in \mathcal{L}(F;E)$. If the operator ideal $\cal
I$ has the Grothendieck property, then $T\in \mathcal{L}_{\cal
I}(F;E)$ if and only if $J_K\in \mathcal{L}_{\mathcal{I}}(E_K;E)$.
\end{proposition}

\begin{proof}
 Assume that $T\in \mathcal{L}_{\cal I}(F;E)$. In this case we have $T(B_F)\in
C_{\mathcal{I}}(E )$. As, for all $\varepsilon>0$,
$\overline{T(B_F)}\subseteq T(B_F)+\varepsilon B_F$ and $\cal I$ has the
Grothendieck property, we have that $\overline{T(B_F)} \in C_{\mathcal{I}}(E )$, hence $K = \dfrac{1}{\|T\|}\overline{T(B_F)}\in C_{\mathcal{I}}(E )$. Given $\varepsilon > 0$, since $C_K\subseteq a^{n/2}K+a^{-n/2}B_E$ for
every $n$ (see \cite[p.\,331]{OjaIsrael}), choosing $n$ such that $a^{-n/2} < \varepsilon$ and putting $A_\varepsilon = a^{n/2}K \in C_{\mathcal{I}}(E )$, we have $C_K\subseteq A_\varepsilon +\varepsilon B_E$. The Grothendieck property of $\cal I$ gives $C_K \in C_{\mathcal{I}}(E )$. By items (b) and (d) of Lemma \ref{lnofl} it follows that
$$J_K(B_{E_K}) = J_K(C_K)=C_K \in C_{\mathcal{I}}(E ),$$ which proves that $J_K \in
  \mathcal{L}_{\mathcal{I}}(E_K;E)$. The converse follows immediately from the ideal property.
\end{proof}

\begin{corollary}Let $T=J_K\circ T_K$ be the LNO factorization of the operator $T\in \mathcal{L}(F;E)$. If the operator ideal $\cal
I$ is surjective and has the Grothendieck property (in particular, if $\cal I$ is closed and surjective), then $T\in \mathcal{L}_{\cal
I}(F;E)$ if and only if $J_K\in \mathcal{L}_{\mathcal{I}}(E_K;E)$.
\end{corollary}

\begin{proof} By \cite[Proposition 3.2]{aronruedaprims} we know that $\mathcal{L}_{\cal I} =\mathcal{I}$ if and only if $\cal I$ is surjective. The {\it in particular} part follows from Example \ref{exGG}.
\end{proof}

The next result, which is a variant of Delgado, Oja, Pi\~neiro and Serrano \cite[Theorem 2.1]{delgadooja} and a generalization of Choi, Kim and Lee \cite[Theorem 2.4]{choikimlee} (see Example \ref{ckl}), shows that with additional assumptions the conditions (a)-(d) of Proposition \ref{prop00} are equivalent.


\begin{theorem} \label{theoequiv} Let ${\cal I}, {\cal J}_1, {\cal J}_2$ be operator ideals such that ${\cal J}_1$ has the Grothendieck property, ${\cal I} \supseteq
\mathcal{L}_{\mathcal{J}_1}=\mathcal{L}_{\mathcal{J}_1}\circ
\mathcal{L}_{\mathcal{J}_2}$ and such that operators belonging to $\cal I$ map
$\mathcal{J}_2$-bounded sets to $\mathcal{J}_1$-bounded sets.
 The following statements are equivalent for a Banach space $E$:\\
{\rm (a)} ${\rm id}_E \in \overline{\mathcal{F}(E;E)}^{\tau_{C_{\mathcal {J}_1}}}$.\\
{\rm (b)}  $E$ has the $(\mathcal{L}_{\mathcal{J}_1}, {\mathcal F}, \|\cdot\|)$-AP\\
{\rm (c)} $E$ has the $(\mathcal{L}_{\mathcal{J}_1}, {\cal F},\tau_{C_{\mathcal {J}_2}} )$-AP.\\
{\rm (d)} $E$ has the $({\cal
I}, {\cal F},\tau_{C_{\mathcal {J}_2}} )$-AP.

\end{theorem}

\begin{proof}

(a) $\Longrightarrow $ (b)  Let $F$ be a Banach space and $T\in
\mathcal{L}_{\mathcal{J}_1}(F;E)$. Since $T$ maps bounded sets to
${\mathcal {J}_1}$-bounded sets and, by assumption, ${\rm id}_E \in \overline{\mathcal{F}(E;E)}^{\tau_{C_{\mathcal {J}_1}}}$, Lemma \ref{lemmaaderencia} yields that $T = {\rm id}_E \circ T \in \overline{\mathcal{F}(F;E)}^{\|\cdot\|}$.\\
(b) $\Longrightarrow $ (a) Let $A\in C_{{\cal J}_1}(E)$ and $\varepsilon>0$ be given. There exists a Banach space $F$ and an operator $T\in \mathcal{J}_1(F;E) \subseteq \mathcal{L}_{\mathcal{J}_1}(F;E)$
  such that   $A \subseteq  T(B_F)$. Lettting $T=J_K\circ T_K$ be the LNO factorization of $T$, $J_K\in \mathcal{L}_{\mathcal{J}_1}(E_K;E)$  by Proposition \ref{lemmajk} as ${\cal J}_1$
  has the Grothendieck property. By assumption there exists an operator $S\in \mathcal{F}(E_K, E)$, say $S=\sum_{i=1}^{n} y_i^*\otimes
    x_i$ with $y_1^*, \ldots, y_n^*\in E'_K$ and $x_1, \ldots, x_n\in E$,
   such that $\|S - J_K\|<\dfrac{\varepsilon}{2\|T\|}$.
   We know that $J''$ is injective (Lemma \ref{lnofl}(c)), so
     $\overline{J'_K(E')}^{\|.\|}=E'_K$, thus for each $i = 1, \ldots, n$, there exists  $x^*_i\in
      E'$  such that $$\|y^*_i-J_K(x^*_i)\|<\dfrac{\varepsilon}{2 \|T\|\cdot\displaystyle\sum_{i=1}^n\|x_i\|}. $$
Let $R= \displaystyle\sum_{i=1}^{n} x_i^*\otimes
    x_i \in \mathcal{F}(E, E)$. For every $x\in K$, as $J_K(x)=x$ and $K \subseteq C_K = B_{E_K}$ (Lemma \ref{lnofl}(a),(b)), we have
\begin{align*}
\|R(x)-x\|&=\|R(J_K( x))-J_K( x)\|\leq \|(R\circ J_K)(x)-S(x)\|+ \|S(x)-J_K( x)\| \\
&\leq \|R\circ J_K-S\|\cdot\|x\|_K+ \|S-J_K\|\cdot\| x\|_K \leq \|R\circ J_K-S\|+ \|S-J_K\|\\  
&=  \left\| \displaystyle\sum_{i=1}^{n} (x_i^*\circ J_K)\otimes
    x_i- \displaystyle\sum_{i=1}^{n} y_i^*\otimes
    x_i\right\| + \|S-J_K\| \\
&= \left\|\displaystyle\sum_{i=1}^{n} \left(J'_K( x_i^*)-y_i^* \right)\otimes
    x_i\right\| + \|S-J_K\| \\
&\leq  \displaystyle\sum_{i=1}^{n}\| J'_K (x_i^*)-y_i^* \|\cdot\|x_i\|+ \|S-J_K\| \\
&<
\dfrac{\varepsilon}{2\|T\|\cdot\displaystyle\sum_{i=1}^n\|x_i\|}\cdot
\displaystyle\sum_{i=1}^n\|x_i\|+
\dfrac{\varepsilon}{2\|T\|}=\dfrac{\varepsilon}{\|T\|}.
\end{align*}
Since $A\subseteq T(B_F) \subseteq \|T\|K$, we have
$$\|R - {\rm id}_E\|_A \leq \|T\|\cdot \|R - {\rm id}_E\|_K < \varepsilon, $$
which proves that $ {\rm id}_E \in
\overline{\mathcal{F}(E;E)}^{\tau_{C_{\mathcal {J}_1}}}$.\\
 (b) $\Longrightarrow $(c)  The same argument of the proof of (b) $\Longrightarrow $(c) in Proposition \ref{prop00} works.\\
(c) $\Longrightarrow$ (b)
 Let $F$ be a Banach space and $T\in
\mathcal{L}_{\mathcal{J}_1}(F;E)$. As $
\mathcal{L}_{\mathcal{J}_1}=\mathcal{L}_{\mathcal{J}_1}\circ
\mathcal{L}_{\mathcal{J}_2}$, there exist a Banach space $Z$ and operators $\mathcal{L}_{\mathcal{J}_1}(Z;E)$ and $S\in
\mathcal{L}_{\mathcal{J}_2}(F;Z)$ such that $ T= R\circ S$. By assumption,
$R\in \overline{{\cal F}(Z;E)} ^{\tau_{C_{\mathcal {J}_2}}}$.
As $S$ maps bounded sets to
 $\mathcal{J}_2$-bounded sets, by Lemma \ref{lemmaaderencia} it follows that $ T= R\circ S\in
 \overline{\mathcal{F}(F;E)}^{\|\cdot\|}$.\\
(d) $\Longrightarrow$ (c)  It follows from the relation $
\mathcal{L}_{\mathcal{J}_1}\subseteq \cal I$ .\\
(a) $\Longrightarrow$ (d) The assumption that operators in $\cal I$ map
$\mathcal{J}_2$-bounded sets to $\mathcal{J}_1$-bounded sets allows us to repeat the argument of the proof of (a) $\Longrightarrow $(b).
\end{proof}

\begin{remark}\rm Observe that the Grothendieck property of ${\cal J}_1$ is used only in the proof of
(b) $\Longrightarrow$ (a) and the condition $\mathcal{L}_{\mathcal{J}_1}=\mathcal{L}_{\mathcal{J}_1}\circ
\mathcal{L}_{\mathcal{J}_2}$ is used only in the proof of (c) $\Longrightarrow$ (b).
\end{remark}

Let us see that Theorem \ref{theoequiv} recovers a result due to Choi, Kim and Lee \cite{choikimlee} as a particular case

\begin{example}\label{ckl}\rm Let ${\cal
J}_1=\mathcal{K}$, ${\cal J}_2= \cal K$ and $\mathcal I$ be an
operator ideal containing $\mathcal K$. Since the ideal $\cal K$ is
surjective and closed, it enjoys the Grothendieck property (cf. Example \ref{exGG}). It is well known that $\cal K= \cal K\circ \cal K$ (cf. the proof of \cite[Proposition 3.1.3]{pietsch} or \cite[Theorem 2.2]{OjaIsrael}), so by Theorem \ref{theoequiv}, the following statements are equivalent for a Banach space $E$:\\
(a) ${\rm id}_E \in \overline{\mathcal{F}(E;E)}^{\tau_c}$\\
(b)  $E$ has the $(\mathcal{K}, {\mathcal F}, \|\cdot\|)$-AP\\
 (c) $E$ has the $(\mathcal{K}, {\cal F},\tau_c )$-AP.\\
(d) $E$ has the $(\mathcal{I}, {\cal F},\tau_c )$-AP.

This recovers \cite[Theorem 2.4]{choikimlee}.
\end{example}

\section{Projective ideal topologies}\label{ptp}
In the previous section we showed that known results on approximation properties in Banach spaces can be recovered as particular instances of more general results in context of the approach to approximation properties by means of ideal topologies we propose in this paper. In this section we reinforce this unifying feature of our approach by proving that some recent results of \cite{erhanpilar,BB} on approximation properties in projective tensor products of Banach spaces are particular instances of much more general results in the realm of ideal topologies. It is worth noticing that two results of \c Caliskan and Rueda \cite{erhanpilar}, one for the weak approximation property of Choi and Kim \cite{Choi-Kim 2006} (cf. Example \ref{primex}(d)) and another one for the quasi approximation property (cf. Example \ref{primex}(e)), are in fact particular instances of one single result.

Our interest in approximation properties in projective tensor products of Banach spaces (remember that approximation properties and topological tensor products are closely connected since Grothendieck \cite{grothendieck}) leads us to the following refinement of the definition of ideal topology:

\begin{definition}\label{def2}\rm Let $\cal C$ be class of Banach spaces, that is, a subclass of BAN. A {\it $\cal C$-projective ideal topology} $\tau$ is a correspondence that, for all positive integers $n \in \mathbb{N}$ and Banach spaces $E, E_1, \ldots,E_n$ and $F$, assigns a linear topology, still denoted by $\tau$, on each of the following spaces: ${\cal L}(E;F)$, ${\cal P}(^n E;F)$ and ${\cal L}(E_1, \ldots, E_n;F)$; such that:\\
(i) When restricted to the spaces ${\cal L}(E;F)$, $\tau$ is an ideal topology.\\
(ii) For all $n \in \mathbb{N}$ and Banach spaces $E, E_1, \ldots, E_n,F$ with $E$ and at least one of the $E_j$ belonging to $\cal C$, the linear bijections
$$P \in \left({\cal P}(^n E;F), \tau \right) \mapsto P_L \in \left({\cal L}\left(\widehat{\otimes}_{s,\pi}^n E ;F\right) , \tau\right){\rm ~and} $$
$$A \in \left({\cal L}(E_1, \ldots, E_n;F), \tau \right) \mapsto A_L \in \left({\cal L}\left(E_1\widehat{\otimes}_{\pi}\cdots  \widehat{\otimes}_{\pi}E_n;F\right) , \tau\right) $$
are homeomorphisms. For simplicity, a BAN-projective ideal topology shall be referred to as a projective ideal topology.
\end{definition}

It is well known that the norm topology is a projective ideal topology. Let us see that the topology of pointwise convergente is a projective ideal topology as well:

\begin{proposition}\label{tcptip} The topology of pointwise convergence $\tau_P$ is a projective ideal topology.
\end{proposition}

\begin{proof}  We already know that $\tau_P$ is an ideal topology (Example \ref{ex1}(b)).
Let $(P_\lambda)_\lambda$ be a net in ${\cal P}(^n E;F)$ such that $P_\lambda \stackrel{\tau_P}{\longrightarrow} P \in {\cal P}(^n E;F)$. We have to prove that $(P_\lambda)_L \stackrel{\tau_P}{\longrightarrow} P_L$ in ${\cal L}\left(\widehat{\otimes}_{s,\pi}^n E ;F\right)$, that is, $(P_\lambda)_L (z) \longrightarrow P_L(z)$ in $F$ for every $z \in \widehat{\otimes}_{s,\pi}^n E$. Assume first that $z = \sum_{j=1}^k \lambda_j \otimes^n x_j$ for some $k \in \mathbb{N}$, $x_1, \ldots, x_k \in E$ and nonzero scalars $\lambda_1, \ldots, \lambda_k \in \mathbb{K}$. Given $\varepsilon > 0$, there exists $\lambda_0$ such that
$$\|P_\lambda - P\|_{\{x_1, \ldots, x_k\}} < \frac{\varepsilon}{k\cdot \max\limits_{j=1,\ldots, k}|\lambda_j|}{\rm~ ~for~every~} \lambda \geq \lambda_0. $$
So, for $\lambda \geq \lambda_0$,
\begin{align*} \|(P_\lambda)_L (z) - P_L(z)\|&= \left\|(P_\lambda - P)_L \left(\sum_{j=1}^k \lambda_j \otimes^n x_j\right) \right\|\leq \sum_{j=1}^k  \left\|(P_\lambda - P)_L \left(\lambda_j \otimes^n x_j\right) \right\|\\
&= \sum_{j=1}^k  |\lambda_j|\cdot \left\|(P_\lambda - P)_L \left(\otimes^n x_j\right) \right\| = \sum_{j=1}^k  |\lambda_j|\cdot \left\|(P_\lambda - P) \left(x_j\right) \right\| < \varepsilon.
\end{align*}
This proves that $(P_\lambda)_L (z) \longrightarrow P_L(z)$ in $F$. Observe that $\left(P_\lambda - P\right)_\lambda$ is collection of continuous $n$-homogeneous polynomials from the Banach space $\widehat{\otimes}_{s,\pi}^n E$ to the Banach space $F$. The convergence $P_\lambda \stackrel{\tau_P}{\longrightarrow} P$ implies, in particular, that the collection $\left(P_\lambda - P\right)_\lambda$ is pointwise bounded, so by the polynomial version of the Banach--Steinhaus Theorem \cite[Theorem 2.6]{livromujica} it follows that it is norm bounded, that is, there is a constant $K>0$ such that $\|P_\lambda - P\| \leq K$ for every $\lambda$. Let now $z$ be an arbitrary element of $\widehat{\otimes}_{s,\pi}^n E $. There are sequences $(x_j)_{j=1}^\infty$ in $E$ and $(\lambda_j)_{j=1}^\infty$ in $\mathbb{K}$ such that
$$z = \sum_{j=1}^\infty \lambda_j \otimes^n x_j {\rm ~~and~~} \sum_{j=1}^\infty |\lambda_j|\cdot\| x_j\|^n < \infty $$
(see \cite[Proposition 2.2(9)]{klaus}). Given $\varepsilon > 0$, let $n_0$ be such that $\sum_{j= n_0+1}^\infty |\lambda_j|\cdot\| x_j\|^n< \frac{\varepsilon}{2K}$. Calling $z' = \sum_{j=1}^{n_0} \lambda_j \otimes^n x_j$, by the first part of the proof we know that $(P_\lambda)_L (z') \longrightarrow P_L(z')$ in $F$. Let $\lambda_0$ be such that $\|(P_\lambda)_L(z') - P_L(z')\| < \frac{\varepsilon}{2}$ whenever $\lambda \geq \lambda_0$. Thus,
\begin{align*} \|(P_\lambda)_L (z) - P_L(z)\|&= \left\|(P_\lambda - P)_L \left(\sum_{j=1}^\infty \lambda_j \otimes^n x_j\right) \right\| \\&= \left\|(P_\lambda - P)_L \left(\sum_{j=1}^{n_0} \lambda_j \otimes^n x_j\right) + (P_\lambda - P)_L \left(\sum_{j=n_0+1}^\infty \lambda_j \otimes^n x_j\right) \right\|\\& \leq \left\|(P_\lambda - P)_L \left(\sum_{j=1}^{n_0} \lambda_j \otimes^n x_j\right)\right\| +\sum_{j=n_0+1}^\infty \left\|(P_\lambda - P)_L \left( \lambda_j \otimes^n x_j\right) \right\|\\& < \frac{\varepsilon}{2} + \sum_{j=n_0+1}^\infty |\lambda_j|\cdot \left\|(P_\lambda - P) \left( x_j\right) \right\|\\&\leq \frac{\varepsilon}{2} + \sum_{j=n_0+1}^\infty |\lambda_j|\cdot \left\|P_\lambda - P\right\|\cdot \left\| x_j\right\|^n < \varepsilon,
\end{align*}
for every $\lambda \geq \lambda_0$, proving that $(P_\lambda)_L (z) \longrightarrow P_L(z)$ in $F$.

The converse is easy. Given a net $(u_\lambda)_\lambda$ in ${\cal L}\left(\widehat{\otimes}_{s,\pi}^n E ;F\right)$ such that $u_\lambda \stackrel{\tau_P}{\longrightarrow} u \in {\cal L}\left(\widehat{\otimes}_{s,\pi}^n E ;F\right)$, there are (unique) polynomials $(P_\lambda)_\lambda$ and $P$ in ${\cal P}(^n E;F)$ such that $(P_\lambda)_L = u_\lambda$ for every $\lambda$ and $P_L = u$. For every $x \in E$,
$$P_\lambda(x)   = (P_\lambda)_L(\otimes^n x) = u_\lambda(\otimes^n x) \longrightarrow u(\otimes^n x) = P_L( \otimes^nx) = P(x).$$
This proves that $P_\lambda \stackrel{\tau_P}{\longrightarrow} P$ and completes the proof of the polynomial case of condition \ref{def2}(ii). The multilinear case is analogous (for a simple proof of the multilinear Banach--Steinhaus Theorem, see Bernardino \cite{Thiago}).
\end{proof}

Now let us give some further examples of projective ideal topologies that are useful in the study of the approximation properties. For $A \subseteq E$ and $A_j \subseteq E_j$, $j = 1, \ldots, n$, define
$$A_1 \otimes \cdots \otimes A_n := \{x_1 \otimes \cdots \otimes x_n : x_j \in A_j, j= 1, \ldots, n\}\subseteq E_1 \otimes \cdots \otimes E_n, $$
$$\otimes_s^n A := \{\otimes^n x : x \in A\} \subseteq \otimes_s^n E. $$

\begin{proposition}\label{proppp} Let ${\cal C} \subseteq {\rm BAN}$ be given. Suppose that for every Banach space $E$ it has been assigned a collection ${\cal A}(E)$ of bounded subsets of $E$ containing the singletons, satisfying {\rm (\ref{cond})} and such that, for all $n \in \mathbb{N}$ and Banach spaces $E_1, \ldots, E_n,E$ with $E,E_j \in {\cal C}$ for some $j$, the following hold:\\
{\rm (i)} Every $A \in {\cal A}(E_1\widehat{\otimes}_{\pi}\cdots  \widehat{\otimes}_{\pi}E_n)$ is contained in a finite union of sets of the form $\overline{{\rm co}}(A_1 \otimes \cdots \otimes A_n)$, where $A_j \in {\cal A}(E_j)$, $j = 1, \ldots, n$.\\
{\rm (ii)} If $A_j \in {\cal A}(E_j)$ for $j = 1, \ldots, n$, then there is $A \in  {\cal A}(E_1\widehat{\otimes}_{\pi}\cdots  \widehat{\otimes}_{\pi}E_n)$ such that $A_1 \otimes \cdots \otimes A_n \subseteq A$.  \\
{\rm (iii)}  Every $A \in {\cal A}\left(\widehat{\otimes}_{s,\pi}^n E \right)$ is contained in a finite union of sets of the form $\overline{{\rm co}}(\otimes_s^nA')$, where $A' \in {\cal A}(E)$. \\
{\rm (iv)} If $A \in {\cal A}(E)$, then there is $A' \in {\cal A}\left(\widehat{\otimes}_{s,\pi}^n E \right) $ such that $\otimes_s^nA \subseteq A'$.\\
By $\tau_{\cal A}$ we mean the topology on the spaces ${\cal L}(E;F)$ and ${\cal P}(^nE;F)$ of uniform convergence on sets of ${\cal A}(E)$, and the topology on the space ${\cal L}(E_1, \ldots, E_n;F)$ of uniform convergence on sets of ${\cal A}(E_1) \times \cdots \times {\cal A}(E_n)$. Then $\tau_{\cal A}$ is a $\cal C$-projective ideal topology.
\end{proposition}

\begin{proof} We already know that $\tau_{\cal A}$ is an ideal topology (Proposition \ref{propger}). Let $E$ and $F$ be Banach spaces with $E \in {\cal C}$ and let $(P_\lambda)_\lambda$ be a net in ${\cal P}(^n E;F)$ such that $P_\lambda \stackrel{\tau_{\cal A}}{\longrightarrow} P \in {\cal P}(^n E;F)$. Let $A  \in {\cal A}\left(\widehat{\otimes}_{s,\pi}^n E\right)$ and $\varepsilon > 0$. By condition (iii) there exist $k \in \mathbb{N}$ and sets $A_1', \ldots, A_k' \in {\cal A}(E)$ such that $A \subseteq \bigcup\limits_{j=1}^k\left( \overline{{\rm co}}(\otimes_n^sA_j')\right)$. Let $\lambda_0$ be such that $\|P_\lambda - P\|_{A_j'} < \varepsilon$, $j = 1, \ldots, k$, whenever $\lambda \geq \lambda_0$. Since $(P_\lambda)_L$ and $P_L$ are continuous linear operators,
\begin{align*}\|(P_\lambda)_L - P_L\|_A &\leq \|(P_\lambda)_L - P_L\|_{\bigcup\limits_{j=1}^k\left( \overline{{\rm co}}(\otimes_n^sA_j')\right)} = \max_{j=1, \ldots, k} \|(P_\lambda)_L - P_L\|_{{\overline{\rm co}}(\otimes_n^sA_j')} = \\& = \max_{j=1, \ldots, k} \|(P_\lambda)_L - P_L\|_{\otimes_n^sA_j'} = \max_{j=1, \ldots, k} \|P_\lambda - P\|_{A'_j} < \varepsilon
\end{align*}
whenever $\lambda \geq \lambda_0$. This proves that $(P_\lambda)_L \stackrel{\tau_c}{\longrightarrow} P_L$ in ${\cal L}\left(\widehat{\otimes}_{s,\pi}^n E ;F\right)$.

Conversely, let $(u_\lambda)_\lambda$ be a net in ${\cal L}\left(\widehat{\otimes}_{s,\pi}^n E ;F\right)$ such that $u_\lambda \stackrel{\tau_{\cal A}}{\longrightarrow} u \in {\cal L}\left(\widehat{\otimes}_{s,\pi}^n E ;F\right)$. There are $(P_\lambda)_\lambda$ and $P$ in ${\cal P}(^n E;F)$ such that $(P_\lambda)_L = u_\lambda$ for every $\lambda$ and $P_L = u$. Let $A \in {\cal A}(E)$ and $\varepsilon > 0$. By condition (iv) there is a set $A' \in {\cal A}\left(\widehat{\otimes}_{s,\pi}^n E\right)$ such that $\otimes_n^s(A) \subseteq A'$. So there is $\lambda_0$ such that $\|u_\lambda - u\|_{A'} < \varepsilon$ for $\lambda \geq \lambda_0$. Thus,
$$\|P_\lambda - P\|_A = \|(P_\lambda)_L - P_L\|_{\otimes_n^sA} = \|u_\lambda - u\|_{\otimes_n^s A} \leq \|u_\lambda - u\|_{A'} < \varepsilon,  $$
for $\lambda \geq \lambda_0$. This proves that $P_\lambda \stackrel{\tau_{\cal A}}{\longrightarrow} P$ and completes the proof of the polynomial case of condition \ref{def2}(ii). The multilinear case is analogous.
\end{proof}

\begin{example}\label{excapai}\rm Choosing ${\cal A}(E)$ as the collection of compact subsets of the Banach space $E$, it is well known that the conditions of Proposition \ref{proppp} are fulfilled. Indeed, condition {\rm (\ref{cond})} is obvious; every compact subset of $E_1\widehat{\otimes}_{\pi}\cdots  \widehat{\otimes}_{\pi}E_n$ is contained in a set of the form $\overline{{\rm co}}(A_1 \otimes \cdots \otimes A_n)$, where $A_j$ is compact in $E_j$ for $j = 1, \ldots, n$ (see Diestel and Puglisi \cite[Proposition 2.1]{diestelpuglisi}); and $K_1 \otimes \cdots \otimes K_n$ is compact in $E_1\widehat{\otimes}_{\pi}\cdots  \widehat{\otimes}_{\pi}E_n$ whenever $K_j$ is compact in $E_j$, $j = 1, \ldots, n$ \cite[p.\,509]{diestelpuglisi}. So letting $\tau_c$ be the compact-open topology on the spaces ${\cal L}(E;F)$ and ${\cal P}(^nE;F)$ and the topology on the space ${\cal L}(E_1, \ldots, E_n;F)$ of uniform convergent on cartesian products of compact sets, we have by Proposition \ref{proppp} that $\tau_c$ is a projective ideal topology.
\end{example}

\begin{example}\label{excapai2}\rm For every Banach space $E$, let ${\cal A}(E)$ be the collection of convex compact subsets of $E$. Trivially, $\cal A$ satisfies condition {\rm (\ref{cond})}. As to condition \ref{proppp}(i), given a compact convex set $A \in E_1\widehat{\otimes}_{\pi}\cdots  \widehat{\otimes}_{\pi}E_n$, as in Example \ref{excapai} there are compact sets $A_j \subseteq E_j$, $j = 1, \ldots n$, such that $A \subseteq \overline{{\rm co}}(A_1 \otimes \cdots \otimes A_n)$. Then each $\overline{\rm co}(A_j)$ is compact and convex in $E_j$ by Mazur's Compactness Theorem \cite[Theorem 2.8.15]{megginson} and $A \subseteq \overline{{\rm co}}(\overline{{\rm co}}(A_1) \otimes \cdots \otimes \overline{{\rm co}}(A_n))$. As to condition \ref{proppp}(ii), given convex compacts sets $K_j \subseteq E_j$, $j = 1,\ldots, n$, as in Example  \ref{excapai} we know that $K_1 \otimes \cdots \otimes K_n$ is compact in $E_1\widehat{\otimes}_{\pi}\cdots  \widehat{\otimes}_{\pi}E_n$. By Mazur's Theorem we have that $\overline{\rm co}(K_1 \otimes \cdots \otimes K_n)$ is a compact convex set containing $K_1 \otimes \cdots \otimes K_n$. By Proposition \ref{proppp}, if $\tau_{\cal A}$ is the topology on spaces of linear operators and polynomials of uniform convergence on compact convex sets and the topology on spaces of multilinear mappings of uniform convergence on cartesian products of compact convex sets, then $\tau_{\cal A}$ is a projective ideal topology.
\end{example}

\begin{example}\label{excapai3}\rm For every Banach space $E$, let ${\cal A}(E)$ be the collection of weakly compact subsets of $E$. Since bounded linear operators are weak-weak continuous, $\cal A$ satisfies condition {\rm (\ref{cond})}. Diestel and Puglisi \cite[Theorem 3.1]{diestelpuglisi} guarantees that condition \ref{proppp}(i) is fulfilled for all Banach spaces. Let ${\cal DP}$ be the class of all Banach spaces with the Dunford-Pettis property. By \cite[Proposition 2.5]{diestelpuglisi} we know that if $A_j$ is weakly compact in $E_j$, $j = 1, \ldots, n$, and some of the $E_j$ has the Dunford-Pettis property, then $A_1 \otimes \cdots \otimes A_n$ is weakly compact in $ E_1\widehat{\otimes}_{\pi}\cdots  \widehat{\otimes}_{\pi}E_n$. This proves condition \ref{proppp}(ii) and shows that the topology of uniform convergence on weakly compact sets or on products of weakly compact sets is a ${\cal DP}$-projective ideal topology.
\end{example}

\begin{example}\label{excapai4}\rm For every Banach space $E$, let ${\cal A}(E)$ be the collection of convex weakly compact subsets of $E$. As before, $\cal A$ satisfies condition {\rm (\ref{cond})}. Applying \cite[Theorem 3.1]{diestelpuglisi} together with the Krein-Smulian Theorem (the closed convex hull of a weakly compact subset of a Banach space is weakly compact as well) we have that condition \ref{proppp}(i) is fulfilled for all Banach spaces. Let ${\cal DP}$ be the class of all Banach spaces with the Dunford-Pettis property. Applying \cite[Proposition 2.5]{diestelpuglisi} together with the same Krein-Smulian Theorem we have that  condition \ref{proppp}(ii) is satisfied for the class $\cal DP$. So the topology of uniform convergence on convex weakly compact sets or on products of convex weakly compact sets is a ${\cal DP}$-projective ideal topology.
\end{example}

Let us put the projective ideal topologies to work. Our first aim is to generalize the results of \c Caliskan and Rueda \cite[Section 3]{erhanpilar}.

We need the notion of composition multi-ideals and composition polynomial ideals:

\begin{definition}\rm Let $\cal I$ be an operator ideal. We say that:\\
(a) A multilinear mapping $A \in {\cal L}(E_1, \ldots, E_n;F)$ belongs to the composition multi-ideal ${\cal I \circ L}$, in symbols $A \in {\cal I \circ L}(E_1, \ldots, E_n;F)$, if there are Banach spaces $G$, a multilinear mapping $B \in {\cal L}(E_1, \ldots, E_n;G)$ and an operator $u \in {\cal I}(G;F)$ such that $A = u \circ B$.\\
(b) A polynomial $P \in {\cal P}(^n E;F)$ belongs to the composition polynomial ideal ${\cal I \circ P}$, in symbols $P \in {\cal I \circ P}(^n E;F)$, if there are a Banach space $G$, a polynomial $Q \in {\cal P}(^n E;G)$ and an operator $u \in {\cal I}(G;F)$ such that $P = u \circ Q$.
\end{definition}

Further details on these  polynomial/multi-ideals can be found in \cite{PRIMS}.

\begin{proposition}\label{prop1} Let ${\cal I}, {\cal J}$ be operator ideals, ${\cal C}\subseteq {\rm BAN}$, $\tau$ be a $\cal C$-projective ideal topology, $n \in \mathbb{N}$ and $E,F$ be Banach spaces with $E \in {\cal C}$. Consider the following conditions:
\begin{enumerate}
\item[\rm(a)] $\mathcal{I}\left(\widehat{\otimes}^{n}_{s,\pi}E;
F\right)\subseteq \overline{\mathcal
{J}\left(\widehat{\otimes}^{n}_{s,\pi}E;
F\right)}^{\, \tau}$.
\item[\rm (b)] $\mathcal{I}\circ \mathcal{P}\left(^nE;
F\right)\subseteq \overline{ \mathcal{J}\circ
\mathcal{P}(^nE; F)}^{\, \tau}$.
\item[\rm (c)] $\mathcal{I} \left(E; F\right)\subseteq
\overline{\mathcal {J}(E;F)}^{\, \tau}$.
\end{enumerate}
Then {\rm (a)} and {\rm (b)} are equivalent and they imply {\rm (c)}.
\end{proposition}

\begin{proof} Let $L \colon \left({\cal P}(^n E;F), \tau \right) \longrightarrow \left({\cal L}\left(\widehat{\otimes}_{s,\pi}^n E ;F\right) , \tau\right) $ be the linearization operator, that is, $L(P) = P_L$.\\
 (a) $\Longrightarrow$ (b) By \cite[Proposition 3.2]{PRIMS} we know that $L\left({\cal I} \circ {\cal P}(^n E;F)\right) = {\cal I}\left(\widehat{\otimes}_{s,\pi}^n E ;F\right)$ and $L\left({\cal J} \circ {\cal P}(^n E;F)\right) = {\cal J} \left(\widehat{\otimes}_{s,\pi}^n E ;F\right)$. Since $L$ is a homeomorphism, we have
\begin{align*}
\mathcal{I}\circ \mathcal{P}\left(^nE;
F\right)& = L^{-1}\left(L\left(\mathcal{I}\circ \mathcal{P}\left(^nE;
F\right)\right) \right) = L^{-1}\left( {\cal I}\left(\widehat{\otimes}_{s,\pi}^n E ;F\right)\right)\\&\subseteq L^{-1}\left(\overline{\mathcal
{J}\left(\widehat{\otimes}^{n}_{s,\pi}E;
F\right)}^{\, \tau} \right) = \overline{L^{-1}\left(\mathcal
{J}\left(\widehat{\otimes}^{n}_{s,\pi}E;
F\right) \right)}^{\, \tau} = \overline{\mathcal{J}\circ \mathcal{P}\left(^nE;
F\right)}^{\, \tau}.
\end{align*}
(b) $\Longrightarrow$ (a) In the same fashion,
\begin{align*}
\mathcal{I}\left(\widehat{\otimes}^{n}_{s,\pi}E;
F\right)& = L\left(L^{-1}\left(\mathcal{I}\left(\widehat{\otimes}^{n}_{s,\pi}E;
F\right) \right) \right) = L \left(\mathcal{I}\circ \mathcal{P}\left(^nE;
F\right) \right) \subseteq L \left(\overline{ \mathcal{J}\circ
\mathcal{P}(^nE; F)}^{\, \tau} \right)\\
& = \overline{L \left( \mathcal{J}\circ
\mathcal{P}(^nE; F) \right)}^{\, \tau} = \overline{\mathcal
{J}\left(\widehat{\otimes}^{n}_{s,\pi}E;
F\right)}^{\, \tau}.
\end{align*}
(a) $\Longrightarrow$ (c) Let $u \in {\cal I}(E;F)$. It is well known that $\widehat{\otimes}^{n}_{s,\pi}E$ contains a complemented isomorphic copy of $E$ \cite[Corollary 4]{Blasco}, so there are continuous linear operators $j \colon E \longrightarrow \widehat{\otimes}^{n}_{s,\pi}E$ and $p \colon \widehat{\otimes}^{n}_{s,\pi}E \longrightarrow E$ such that $p \circ j = {\rm id}_E$. Then $u \circ p \in {\cal I}\left(\widehat{\otimes}^{n}_{s,\pi}E;F \right)$ and, by assumption, $u \circ p \in \overline{\mathcal
{J}\left(\widehat{\otimes}^{n}_{s,\pi}E;
F\right)}^{\, \tau}$. The ideal property of $\overline{\cal J}^{\, \tau}$ gives $u = u \circ p \circ j \in \overline{\mathcal {J}(E;F)}^{\, \tau}$.
\end{proof}

Making $F = \widehat{\otimes}^{n}_{s,\pi}E$ in Proposition \ref{prop1} we obtain:

\begin{theorem}\label{theo1} Let ${\cal I}, {\cal J}$ be operator ideals, ${\cal C}\subseteq {\rm BAN}$, $\tau$ be a $\cal C$-projective ideal topology, $n \in \mathbb{N}$ and $E \in {\cal C}$. Consider the following conditions:
\begin{enumerate}
\item[\rm (a)] $\widehat{\otimes}^{n}_{s,\pi}E$ has the $(\cal{I}, {\cal J},\tau)$-WAP.
\item[\rm (b)] $\mathcal{I}\circ \mathcal{P}\left(^nE;
\widehat{\otimes}^{n}_{s,\pi}E\right)\subseteq \overline{{\cal J} \circ \mathcal{P}\left(^nE; \widehat{\otimes}^{n}_{s,\pi}E \right)}^{\, \tau}$.
\item[\rm (c)] $\mathcal{I} \left(E; \widehat{\otimes}^{n}_{s,\pi}E \right)\subseteq
\overline{\mathcal {J}\left(E;
\widehat{\otimes}^{n}_{s,\pi}E\right)}^{\, \tau}$.
\end{enumerate}
Then {\rm (a)} and {\rm (b)} are equivalent and they imply {\rm (c)}.
\end{theorem}

We need two ingredients to recover Proposition 7 and Proposition 8 of \cite{erhanpilar} as particular instances of Theorem \ref{theo1}. Remember that a vector space-valued map has {\it finite rank} if its range generates a finite dimensional subspace of the target vector space. It is easy to check that a polynomial $ P \in {\cal P}(^n E;F)$ has finite rank if and only if there are $k \in \mathbb{N}$, $P_1, \ldots, P_k \in {\cal P}(^n E)$ and $b_1, \ldots, b_k \in F$ such that $P = \sum_{j=1}^k P_j \otimes b_j$. The space of all such polynomials is denoted by ${\cal P}_{\cal F}(^n E;F)$. The first ingredient is the following elementary lemma:

\begin{lemma}\label{lemma1} ${\cal F} \circ {\cal P} = {\cal P}_{\cal F}$.
\end{lemma}

\begin{proof} Let $P \in {\cal P}(^n E;F)$. Is is easy to check that $
[P(E)] = P_L\left(\widehat{\otimes}^{n}_{s,\pi}E\right)$. So,
\begin{align*} P \in {\cal F} \circ {\cal P}(^n E;F)& \Longleftrightarrow P_L \in {\cal F} \left(\widehat{\otimes}^{n}_{s,\pi}E ;F\right) \Longleftrightarrow \dim  P_L\left(\widehat{\otimes}^{n}_{s,\pi}E\right) < \infty \\&\Longleftrightarrow \dim [P(E)] < \infty \Longleftrightarrow P \in {\cal P}_{\cal F}(^nE;F),
\end{align*}
where the first equivalence follows from \cite[Proposition 3.2]{PRIMS}.
\end{proof}

Let ${\cal P}_{\cal K}$ denote the class of compact homogeneous polynomials between Banach spaces (bounded sets are sent to relatively compact sets). The second ingredient is a classical result due to Aron and Schottenloher \cite{as} that asserts that
\begin{equation} \label{eq1} {\cal P}_{\cal K} = {\cal K}\circ {\cal P}.
\end{equation}

Making $\tau = \tau_c$, $\cal I = {\cal K}$, ${\cal J} = {\cal F}$ and ${\cal C} = {\rm BAN}$ in Theorem \ref{theo1}, with the help of Lemma \ref{lemma1} and (\ref{eq1}) we get:

\begin{proposition}{\rm (\cite[Proposition 7]{erhanpilar})} Let $n \in \mathbb{N}$ and $E$ be a Banach space. Consider the following conditions:
\begin{enumerate}
\item[\rm (a)] $\widehat{\otimes}^{n}_{s,\pi}E$ has the $({\cal K}, {\cal F},\tau_c)$-WAP.
\item[\rm (b)] $\mathcal{P}_{\cal K}\left(^nE;
\widehat{\otimes}^{n}_{s,\pi}E\right)\subseteq \overline{ \mathcal{P}_{\cal F}\left(^nE; \widehat{\otimes}^{n}_{s,\pi}E \right)}^{\, \tau_c}$.
\item[\rm (c)] $\mathcal{K} \left(E; \widehat{\otimes}^{n}_{s,\pi}E \right)\subseteq
\overline{\mathcal {\cal F}\left(E;
\widehat{\otimes}^{n}_{s,\pi}E\right)}^{\, \tau_c}$.
\end{enumerate}
Then {\rm (a)} and {\rm (b)} are equivalent and they imply {\rm (c)}.
\end{proposition}

\begin{remark}\rm Condition (b) in \cite[Proposition 7]{erhanpilar} reads $\mathcal{P}_{\cal K}\left(^nE;
\widehat{\otimes}^{n}_{s,\pi}E\right) = \overline{ \mathcal{P}_{\cal F}\left(^nE; \widehat{\otimes}^{n}_{s,\pi}E \right)}^{\, \tau_c}$, but a glance at its proof reveals that it should read $\mathcal{P}_{\cal K}\left(^nE;
\widehat{\otimes}^{n}_{s,\pi}E\right)\subseteq \overline{ \mathcal{P}_{\cal F}\left(^nE; \widehat{\otimes}^{n}_{s,\pi}E \right)}^{\, \tau_c}$.
\end{remark}

Making $\tau = \|\cdot\|$, $\cal I = {\cal K}$, ${\cal J} = {\cal F}$ and and ${\cal C} = {\rm BAN}$ in Theorem \ref{theo1}, with the help of Lemma \ref{lemma1} and (\ref{eq1}) and remembering that ${\cal P}_{\cal K}$ and $\cal K$ are norm closed, we get:

\begin{proposition}{\rm (\cite[Proposition 8]{erhanpilar})} Let $n \in \mathbb{N}$ and $E$ be a Banach space. Consider the following conditions:
\begin{enumerate}
\item[\rm (a)] $\widehat{\otimes}^{n}_{s,\pi}E$ has the $({\cal K}, {\cal F},\|\cdot\|)$-WAP.
\item[\rm (b)] $\mathcal{P}_{\cal K}\left(^nE;
\widehat{\otimes}^{n}_{s,\pi}E\right)= \overline{ \mathcal{P}_{\cal F}\left(^nE; \widehat{\otimes}^{n}_{s,\pi}E \right)}^{\, \|\cdot\|}$.
\item[\rm (c)] $\mathcal{K} \left(E; \widehat{\otimes}^{n}_{s,\pi}E \right) =
\overline{\mathcal {\cal F}\left(E;
\widehat{\otimes}^{n}_{s,\pi}E\right)}^{\, \|\cdot\|}$.
\end{enumerate}
Then {\rm (a)} and {\rm (b)} are equivalent and they imply {\rm (c)}.
\end{proposition}

Now we extend the results above to the full projective tensor product. Replacing the projective symmetric tensor product by the projective tensor product, homogeneous polynomials by multilinear mappings and the polynomial ideal ${\cal I} \circ {\cal P}$ by the multi-ideal ${\cal I} \circ {\cal L}$, the proof of Proposition \ref{prop1}, {\it mutatis mutandis}, works. Actually  the proof of the multilinear case is easier, because it is trivial that $E_1\widehat{\otimes}_{\pi}\cdots  \widehat{\otimes}_{\pi}E_n$ contains complemented copies of $E_j$, $j = 1, \ldots, n$. So we have:

\begin{proposition}\label{prop2} Let ${\cal I}, {\cal J}$ be operator ideals, ${\cal C}\subseteq {\rm BAN}$, $\tau$ be $\cal C$-a projective ideal topology, $n \in \mathbb{N}$ and $E_1, \ldots, E_n$, $F$ be Banach spaces with $E_j \in {\cal C}$ for some $j$. Consider the following conditions:
\begin{enumerate}
\item[\rm(a)] $\mathcal{I}\left(E_1\widehat{\otimes}_{\pi}\cdots  \widehat{\otimes}_{\pi}E_n;
F\right)\subseteq \overline{\mathcal
{J}\left(E_1\widehat{\otimes}_{\pi}\cdots  \widehat{\otimes}_{\pi}E_n;
F\right)}^{\, \tau}$.
\item[\rm (b)] $\mathcal{I}\circ \mathcal{L}\left(E_1, \ldots, E_n;
F\right)\subseteq \overline{ \mathcal{J}\circ
\mathcal{L}(E_1, \ldots, E_n; F)}^{\, \tau}$.
\item[\rm (c)] $\mathcal{I} \left(E_j; F\right)\subseteq
\overline{\mathcal {J}(E_j;F)}^{\, \tau}$ for $j = 1, \ldots, n$.
\end{enumerate}
Then {\rm (a)} and {\rm (b)} are equivalent and they imply {\rm (c)}.
\end{proposition}

Making $F = E_1\widehat{\otimes}_{\pi}\cdots  \widehat{\otimes}_{\pi}E_n$ in Proposition \ref{prop2} we get:

\begin{theorem}\label{theo2} Let ${\cal I}, {\cal J}$ be operator ideals, ${\cal C}\subseteq {\rm BAN}$, $\tau$ be a $\cal C$-projective ideal topology, $n \in \mathbb{N}$ and $E_1, \ldots, E_n$ be Banach spaces with $E_j \in {\cal C}$ for some $j$. Consider the following conditions:
\begin{enumerate}
\item[\rm (a)] $E_1\widehat{\otimes}_{\pi}\cdots  \widehat{\otimes}_{\pi}E_n$ has the $(\cal{I}, {\cal J},\tau)$-WAP.
\item[\rm (b)] $\mathcal{I}\circ \mathcal{L}\left(E_1, \ldots, E_n;
E_1\widehat{\otimes}_{\pi}\cdots  \widehat{\otimes}_{\pi}E_n\right)\subseteq \overline{{\cal J} \circ \mathcal{L}\left(E_1, \ldots, E_n; E_1\widehat{\otimes}_{\pi}\cdots  \widehat{\otimes}_{\pi}E_n \right)}^{\, \tau}$.
\item[\rm (c)] $\mathcal{I} \left(E_j; E_1\widehat{\otimes}_{\pi}\cdots  \widehat{\otimes}_{\pi}E_n \right)\subseteq
\overline{\mathcal {J}\left(E_j;
E_1\widehat{\otimes}_{\pi}\cdots  \widehat{\otimes}_{\pi}E_n\right)}^{\, \tau}$ for $j = 1, \ldots, n$.
\end{enumerate}
Then {\rm (a)} and {\rm (b)} are equivalent and they imply {\rm (c)}.
\end{theorem}

By ${\cal L}_{\cal F}(E_1, \ldots, E_n;F)$ we denote the subspace of ${\cal L}(E_1, \ldots, E_n;F)$ of all multilinear mappings of finite rank. The same proof of Lemma \ref{lemma1} gives the formula ${\cal L} \circ {\cal F} = {\cal L}_{\cal F}$. Denoting by ${\cal L}_{\cal K}$ the class of compact multilinear mappings, a classical result due to Pe{\l}czy\'nski \cite[Proposition 3]{pelc} gives the formula ${\cal L} \circ {\cal K} = {\cal L}_{\cal K}$. Thus, a multilinear analogue of \cite[Proposition 7]{erhanpilar} is obtained making ${\cal C} = {\rm BAN}$, $\tau = \tau_c$, $\cal I = {\cal K}$ and ${\cal J} = {\cal F}$ in Theorem \ref{theo2}:

\begin{proposition} Let $n \in \mathbb{N}$ and $E_1, \ldots, E_n$ be Banach spaces. Consider the following conditions:
\begin{enumerate}
\item[\rm (a)] $E_1\widehat{\otimes}_{\pi}\cdots  \widehat{\otimes}_{\pi}E_n$ has the $({\cal K}, {\cal F},\tau_c)$-WAP.
\item[\rm (b)] $\mathcal{L}_{\cal K}\left(E_1, \ldots, E_n;
E_1\widehat{\otimes}_{\pi}\cdots  \widehat{\otimes}_{\pi}E_n\right)\subseteq \overline{ \mathcal{L}_{\cal F}\left(E_1, \ldots, E_n; E_1\widehat{\otimes}_{\pi}\cdots  \widehat{\otimes}_{\pi}E_n \right)}^{\, \tau_c}$.
\item[\rm (c)] $\mathcal{K} \left(E_j; E_1\widehat{\otimes}_{\pi}\cdots  \widehat{\otimes}_{\pi}E_n \right)\subseteq
\overline{\mathcal {\cal F}\left(E_j;
E_1\widehat{\otimes}_{\pi}\cdots  \widehat{\otimes}_{\pi}E_n\right)}^{\, \tau_c}$ for $j = 1, \ldots, n$.
\end{enumerate}
Then {\rm (a)} and {\rm (b)} are equivalent and they imply {\rm (c)}.
\end{proposition}

And remembering that ${\cal L}_{\cal K}$ and $\cal K$ are norm closed, making ${\cal C} = {\rm BAN}$, $\tau = \|\cdot\|$, $\cal I = {\cal K}$ and ${\cal J} = {\cal F}$ in Theorem \ref{theo2} we obtain a multilinear analogue of \cite[Proposition 8]{erhanpilar}:

\begin{proposition} Let $n \in \mathbb{N}$ and $E_1, \ldots, E_n$ be Banach spaces. Consider the following conditions:
\begin{enumerate}
\item[\rm (a)] $E_1\widehat{\otimes}_{\pi}\cdots  \widehat{\otimes}_{\pi}E_n $ has the $({\cal K}, {\cal F},\|\cdot\|)$-WAP.
\item[\rm (b)] $\mathcal{L}_{\cal K}\left(E_1, \ldots, E_n;
E_1\widehat{\otimes}_{\pi}\cdots  \widehat{\otimes}_{\pi}E_n \right)= \overline{ \mathcal{L}_{\cal F}\left(E_1, \ldots, E_n; E_1\widehat{\otimes}_{\pi}\cdots  \widehat{\otimes}_{\pi}E_n  \right)}^{\, \|\cdot\|}$.
\item[\rm (c)] $\mathcal{K} \left(E_j; E_1\widehat{\otimes}_{\pi}\cdots  \widehat{\otimes}_{\pi}E_n  \right) =
\overline{\mathcal {\cal F}\left(E_j;
E_1\widehat{\otimes}_{\pi}\cdots  \widehat{\otimes}_{\pi}E_n \right)}^{\, \|\cdot\|}$ for $j = 1, \ldots, n$.
\end{enumerate}
Then {\rm (a)} and {\rm (b)} are equivalent and they imply {\rm (c)}.
\end{proposition}

We finish the paper showing that the concept of projective ideal topology allows us to generalize the results of \cite[Section 3]{BB}. We shall need the so-called factorization method to generate a multi-ideal from a given operator ideal:

\begin{definition}\rm Let $\cal I$ be an operator ideal. We say that a multilinear mapping $A \in {\cal L}(E_1, \ldots, E_n;F)$ belongs to the multi-ideal ${\cal L}[{\cal I}]$, in symbols $A \in {\cal L}[{\cal I}](E_1, \ldots, E_n;F)$, if there are Banach spaces $G_1, \ldots, G_n$, a multilinear mapping $B \in {\cal L}(G_1, \ldots, G_n;F)$ and operators $u_j \in {\cal I}(E_j;G_j)$, $j = 1, \ldots,n$, such that $A = B \circ (u_1, \ldots, u_n)$.
\end{definition}

Further details on these multi-ideals can be found in \cite{note}.

The examples of projective ideal topologies we have been working with are topologies of uniform convergence on subsets (or products of subsets) belonging to a certain class ${\cal A}(E)$ of subsets of the Banach space $E$, $E \in {\rm BAN}$. The following condition is fulfilled by all of them:
\begin{equation}\label{extracondition} {\rm If~} A_1,A_2 \in {\cal A}(E), {\rm ~then~there~is~} A \in {\cal A}(E) {\rm ~such~that~} A_1 \cup A_2 \subseteq A.
\end{equation}

%

\noindent Indeed, it is obvious that the projective ideal topologies of Proposition \ref{tcptip} and Examples \ref{excapai} and \ref{excapai3} fulfill condition (\ref{extracondition}). And using that the closed convex hull of a (weakly) compact set is (weakly) compact we have that the projective ideal topologies of Examples \ref{excapai2} and \ref{excapai4} fulfill condition (\ref{extracondition}) too.  So imposing condition (\ref{extracondition}) we keep all our examples of projective ideal topologies.

Given operator ideals ${\cal I}_1, \ldots, {\cal I}_n$ and Banach spaces $E_1, \ldots, E_n,F$, by
$${\cal I}_1 \otimes \cdots \otimes {\cal I}_n (E_1, \ldots,E_n; F) $$
we denote that set of all $n$-linear mappings $A \in {\cal L}(E_1, \ldots,E_n; F) $ for which there are linear operators $T_j \in {\cal I}_j(E_j;E_j)$, $j =1,\ldots, n$, and an $n$-linear mapping $B \in {\cal L}(E_1, \ldots, E_n;F)$ such that $A = B \circ (T_1 \ldots , T_n)$.

 The next result generalizes \cite[Proposition 3.4]{BB}, which, in its turn, generalizes a classical result due to Heinrich \cite[Theorem 3.]{heinrich}.

\begin{theorem}\label{genBB} Let ${\cal C} \subseteq {\rm BAN}$, $\cal A$ be as in Proposition \ref{proppp} and satisfying (\ref{extracondition}), $\tau_{\cal A}$ be the corresponding $\cal C$-projective ideal topology, ${\cal I}, {\cal I}_1, \ldots, {\cal I}_n, {\cal J}, {\cal J}_1, \ldots, {\cal J}_n$ be operator ideals with ${\cal L}[{\cal I}_1, \ldots, {\cal I}_n] \subseteq \overline{{\cal I}}^{\,\tau_{\cal A}} \circ {\cal L}$ and $E_1, \ldots, E_n$ be Banach spaces one of them belonging to $\cal C$ such that
\begin{equation}\label{lequ}{\cal J}\circ {\cal L}(E_1, \ldots,E_n; E_1\widehat{\otimes}_{\pi}\cdots  \widehat{\otimes}_{\pi}E_n)  \subseteq {\cal J}_1 \otimes \cdots \otimes {\cal J}_n(E_1, \ldots,E_n; E_1\widehat{\otimes}_{\pi}\cdots  \widehat{\otimes}_{\pi}E_n)  .
\end{equation}
If $E_j$ has the $({ {\cal J}_j},{\cal I}_j,{\tau_{\cal A}})$-WAP for $j = 1, \ldots, n$, then $E_1\widehat{\otimes}_{\pi}\cdots  \widehat{\otimes}_{\pi}E_n$ has the $({\cal J},{\cal I},{\tau_{\cal A}})$-WAP.
\end{theorem}

\begin{proof} Let $ T \in {\cal J}(E_1\widehat{\otimes}_{\pi}\cdots  \widehat{\otimes}_{\pi}E_n;E_1\widehat{\otimes}_{\pi}\cdots  \widehat{\otimes}_{\pi}E_n)$. By \cite[Proposition 3.2]{PRIMS} we know that the $n$-linear mapping $B \in {\cal L}(E_1, \ldots, E_n; E_1\widehat{\otimes}_{\pi}\cdots  \widehat{\otimes}_{\pi}E_n)$ such that $B_L = T$ belongs to ${\cal J} \circ {\cal L}$
. By (\ref{lequ}) there are linear operators $T_j \in {\cal J}_j(E_j;E_j)$, $j =1,\ldots, n$, and an $n$-linear mapping $D \in {\cal L}(E_1, \ldots,E_n; E_1\widehat{\otimes}_{\pi}\cdots  \widehat{\otimes}_{\pi}E_n) $ such that $B = D \circ (T_1 ,\ldots , T_n)$. It follows easily that
 $$T = B_L = D_L \circ (T_1 \otimes \cdots \otimes T_n).$$ Given $A \in {\cal A}(E_1 \hat\otimes_\pi \cdots \hat\otimes_\pi E_n)$, by condition \ref{proppp}(i) there are $k \in \mathbb{N}$ and sets $A_j^i \in {\cal A}(E_j)$, $j=1, \ldots, n$, $i = 1, \ldots, k$, such that $A \subseteq \textstyle\bigcup\limits_{i=1}^k \overline{{\rm co}}(A_1^i \otimes \cdots \otimes A_n^i). $  Let $\varepsilon > 0$. By condition (\ref{extracondition}) there are sets $A_j \in {\cal A}(E_j)$ such that  $A_j^1 \cup \cdots \cup A_j^k \subseteq A_j$, $j =1, \ldots, n$. Since sets in ${\cal A}$ are bounded there is $M > 0$ such that $\|x \| \leq M$ for every $x \in A_j$, $j = 1, \ldots, n$. As $E_1$ has the  $({\cal J}_1,{\cal I}_1,{\tau_{\cal A}})$-WAP, there is an operator $u_1 \in {\cal I}_1(E_1;E_1)$ such that $$\|u_1 -T_1\|_{A_1} < \frac{\varepsilon}{4nM^{n-1}\|D\|\cdot\|T_2\| \cdots \|T_n\|}.$$ As $E_2$ has $({\cal J}_2,{\cal I}_2,{\tau_{\cal A}})$-WAP , there is an operator $u_2 \in {\cal I}_2(E_2;E_2)$ such that
$$\|u_2 - T_2\|_{A_2} < \frac{\varepsilon}{4nM^{n-1}\|D\|\cdot\|u_1\|\cdot \|T_3\| \cdots \|T_n\|}.$$ Continuing the process we obtain operators $u_j \in {\cal I}_j(E_j;E_j)$ such that
$$\|u_j - T_j\|_{A_j} < \frac{\varepsilon}{4nM^{n-1}\|D\|\cdot\|u_1\|\cdots \|u_{j-1}\|\cdot \|T_{j+1}\| \cdots \|T_n\|}$$
for $j = 1, \ldots, n$. Performing a computation identical to the one in the proof of \cite[Proposition 3.4]{BB} we conclude that
\begin{equation}\label{equnew}\|u_1 \otimes \cdots \otimes u_n (x_1 \otimes \cdots \otimes x_n) - T_1(x_1) \otimes \cdots\otimes T_n(x_n) \| < \frac{\varepsilon}{4\|D\|}, \end{equation}
for all $x_1 \in A_1, \ldots, x_n \in A_n$. Using that $D_L$, $u_1 \otimes \cdots \otimes u_n$ and $T_1 \otimes \cdots \otimes T_n$ are all continuous linear operators, from (\ref{equnew}) it follows that
\begin{align*}\left\|D_L \circ (u_1 \otimes \cdots \otimes u_n) - T\right\|_{\overline{\rm co}\left(A_1\otimes \cdots \otimes A_n\right)}~~~~~~~~~~~~~~~~~~~~~~~~~~~~~~~~~~~~~~~~~~~~~~~~~~\\
=\left\|D_L \circ (u_1 \otimes \cdots \otimes u_n) - D_L \circ (T_1 \otimes \cdots \otimes T_n)\right\|_{\overline{\rm co}\left(A_1\otimes \cdots \otimes A_n\right)}~~\\
\leq \|D_L\| \cdot \left\|u_1 \otimes \cdots \otimes u_n - T_1 \otimes \cdots \otimes T_n\right\|_{\overline{\rm co}\left(A_1\otimes \cdots \otimes A_n\right)}~~~~~~~~~~~\\
= \|D\|\cdot \left\|u_1 \otimes \cdots \otimes u_n - T\right\|_{A_1\otimes \cdots \otimes A_n} \leq \frac{\varepsilon}{4}~.\,~~~\,~~~~~~~~~~~~~~~~~~~.
\end{align*}
So,
\begin{align*}
\left\|D_L \circ (u_1 \otimes \cdots \otimes u_n) - T \right\|_A & \leq \left\|D_L \circ (u_1 \otimes \cdots \otimes u_n) - T \right\|_{\bigcup\limits_{i=1}^k \overline{{\rm co}}(A_1^i \otimes \cdots \otimes A_n^i)}\\
& =  \max_{i = 1, \ldots, k}\left\|D_L \circ (u_1 \otimes \cdots \otimes u_n) - T \right\|_{ \overline{{\rm co}}(A_1^i \otimes \cdots \otimes A_n^i)}\\
&\leq \left\|D_L \circ (u_1 \otimes \cdots \otimes u_n) - T \right\|_{ \overline{{\rm co}}\left((A_1^1 \cup \cdots \cup A_1^k) \otimes \cdots \otimes (A_n^1 \cup \cdots \cup A_n^k)\right)}\\
& \leq \left\|D_L \circ (u_1 \otimes \cdots \otimes u_n) - T \right\|_{ \overline{{\rm co}}\left(A_1 \otimes \cdots \otimes A_n\right)}\leq \frac{\varepsilon}{4} < \frac{\varepsilon}{2}.
\end{align*}
We know that $\overline{{\cal I}}^{\,\tau_{\cal A}}$ is an operator ideal because $\tau_{\cal A}$ is an ideal topology, so the assumption ${\cal L}[{\cal I}_1, \ldots, {\cal I}_n] \subseteq \overline{{\cal I}}^{\,\tau_{\cal A}} \circ {\cal L}$ together with \cite[Proposition 3.3]{BB} yield that $u_1  \otimes \cdots \otimes u_n$ belongs to $\overline{{\cal I}}^{\,\tau_{\cal A}} (E_1 \hat\otimes_\pi \cdots \hat\otimes_\pi E_n; E_1
\hat\otimes_\pi \cdots \hat\otimes_\pi E_n)$. Calling on the ideal property of $\overline{{\cal I}}^{\,\tau_{\cal A}}$ once again we conclude that $D_L \circ (u_1  \otimes \cdots \otimes u_n)$ belongs to $\overline{{\cal I}}^{\,\tau_{\cal A}} (E_1 \hat\otimes_\pi \cdots \hat\otimes_\pi E_n; E_1
\hat\otimes_\pi \cdots \hat\otimes_\pi E_n)$ as well. So there is $U \in {\cal I} (E_1 \hat\otimes_\pi \cdots \hat\otimes_\pi E_n; E_1
\hat\otimes_\pi \cdots \hat\otimes_\pi E_n)$ such that $$\|U - D_L \circ (u_1 \otimes \cdots \otimes u_n)\|_A < \frac{\varepsilon}{2}.$$ It follows that $\|U - T\|_A < \varepsilon$, which proves that $T \in \overline{{\cal I}(E_1 \hat\otimes_\pi \cdots \hat\otimes_\pi E_n; E_1
\hat\otimes_\pi \cdots \hat\otimes_\pi E_n)}^{\, \tau_{\cal A}}$ and completes the proof.
\end{proof}

When ${\cal I}_1 = \cdots = {\cal I}_n = {\cal I}$ we write ${\cal L}[{\cal I}]: =\bigcup\limits_{n = 1}^\infty{\cal L}[{\cal I}_1, \ldots, {\cal I}_n]$.

\begin{corollary}\label{genBBcor} Let ${\cal C} \subseteq {\rm BAN}$, $\cal A$ be as in Proposition \ref{proppp} and satisfying (\ref{extracondition}), $\tau_{\cal A}$ be the corresponding $\cal C$-projective ideal topology and $\cal I$, $\cal J$ be operator ideals such that $\mathcal{L}[\mathcal{I}] \subseteq \overline{{\cal I}}^{\,\tau_{\cal A}}\circ\mathcal{L}$.
The following are equivalent for a Banach space $E \in {\cal C}$ such that
${\cal J}\circ {\cal L}\left(^nE; \widehat\otimes_{\pi}^n E \right) \subseteq \otimes^n{\cal J}\left(^nE;\widehat\otimes_{\pi}^n E \right)$ for every $n$ (for some $n$, respectively):

\medskip

\noindent {\rm (a)} $E$ has the $({\cal J}, {\cal I}, \tau_{\cal A})$-WAP.\\
\medskip
{\rm (b)} $\widehat\otimes_\pi^n E$ has the $({\cal J}, {\cal I}, \tau_{\cal A})$-WAP for every $n$ ($\widehat\otimes_\pi^k E$ has the $({\cal J}, {\cal I}, \tau_{\cal A})$-WAP for every $k \leq n$, respectively).\\
\medskip
{\rm (c)} $\widehat\otimes_\pi^n E$ has the $({\cal J}, {\cal I}, \tau_{\cal A})$-WAP for some $n$ ($\widehat\otimes_\pi^k E$ has the $({\cal J}, {\cal I}, \tau_{\cal A})$-WAP for some $k \leq n$, respectively).\\
\medskip
{\rm (d)} $\widehat\otimes_\pi^{n,s}E$ has the $({\cal J}, {\cal I}, \tau_{\cal A})$-WAP for every $n$ ($\widehat\otimes_\pi^{k,s}E$ has the $({\cal J}, {\cal I}, \tau_{\cal A})$-WAP for every $k \leq n$, respectively).\\
\medskip
{\rm (e)} $\widehat\otimes_\pi^{n,s}E$ has the $({\cal J}, {\cal I}, \tau_{\cal A})$-WAP for some $n$ ($\widehat\otimes_\pi^{k,s}E$ has the $({\cal J}, {\cal I}, \tau_{\cal A})$-WAP for some $k \leq n$, respectively).
\end{corollary}

\begin{proof} Just repeat the proof of \cite[Corollary 3.8]{BB} using Theorem \ref{genBB} and Proposition \ref{complem}.
\end{proof}

Since ${\rm id}_{E_1\widehat{\otimes}_{\pi}\cdots  \widehat{\otimes}_{\pi}E_n} = {\rm id}_{E_1}\otimes \cdots \otimes {\rm id}_{E_n}$, it is clear that condition (\ref{lequ}) holds for ${\cal J} = {\cal J}_1 = \cdots = {\cal J}_n= {\cal L}$ and every $n$. Thus, making ${\cal C} = {\rm BAN}$, $ {\cal J} = {\cal J}_1 = \cdots = {\cal J}_n= {\cal L}$ and letting ${\cal A}(E)$ be the collection of compact subsets of the Banach space $E$, that is, $\tau_{\cal A} = \tau_c$, Theorem \ref{genBB} recovers \cite[Proposition 3.4]{BB} and Corollary \ref{genBBcor} recovers \cite[Corollary 3.8]{BB} (remember that $({\cal L}, {\cal I}, \tau)$-AP = $({\cal L}, {\cal I}, \tau$)-WAP).

A number of examples of ideals satisfying $\mathcal{L}[\mathcal{I}_1, \ldots, {\cal I}_n] \subseteq {\cal J}\circ\mathcal{L}$ and/or $\mathcal{L}[\mathcal{I}] \subseteq {\cal J}\circ\mathcal{L}$ can be found in \cite[3.5-3.7]{BB}.

\vspace{2em}

%

\noindent Faculdade de Matem\'atica\\
Universidade Federal de
Uberl\^andia\\
38.400-902 -- Uberl\^andia, Brazil\\
e-mails: soniles@famat.ufu.br, botelho@ufu.br.

\end{document}